\documentclass[12pt, a4paper]{amsart}

\usepackage[utf8]{inputenc}
\usepackage[hmargin=30mm, vmargin=25mm, includefoot, twoside]{geometry}
\usepackage[bookmarksopen=true, urlcolor=magenta,urlbordercolor={1 0 0}]{hyperref}
\hypersetup{colorlinks=true,linkcolor=black,anchorcolor=blue,citecolor=magenta,filecolor=blue,urlcolor=cyan,bookmarksnumbered=true}

\usepackage[alphabetic]{amsrefs}

\usepackage{amsfonts,amssymb,verbatim}
\usepackage{latexsym}
\usepackage{mathrsfs}
\usepackage{stmaryrd}
\usepackage{xspace}
\usepackage{enumerate, paralist}
\usepackage{graphicx}
\usepackage[all,cmtip]{xy}
\usepackage[all]{xypic}

\usepackage{dutchcal}

\usepackage[capitalize]{cleveref}

\usepackage[usenames,dvipsnames]{xcolor}

\usepackage{txfonts, pxfonts}

\newtheorem{cor}{Corollary}[section]

\newtheorem{te}[cor]{Theorem}
\newtheorem{p}[cor]{Proposition}
\newtheorem{qu}[cor]{Question}
\newtheorem{lemma}[cor]{Lemma}

\theoremstyle{definition}
\newtheorem{de}[cor]{Definition}
\newtheorem{ob}[cor]{Observation}
\theoremstyle{remark}

\newtheorem{ex}[cor]{Example}

\newtheorem{nt}[cor]{Notation}

\newcommand{\cz}{\mathbb{C}}
\newcommand{\nz}{\mathbb{N}}
\newcommand{\zz}{\mathbb{Z}}

\newcommand{\rz}{\mathbb{R}}

\newcommand{\acal}{\mathcal{A}}
\newcommand{\ff}{\mathbb{F}}

\newcommand{\gr}{\mathcal{G}}

\newcommand{\tc}{\mathcal{T}}

\newcommand{\mm}{\mathcal{M}}

\newcommand{\unit}{\mathcal{U}}

\newcommand{\s}{Sym}

\newcommand{\vp}{\varphi}
\newcommand{\ve}{\varepsilon}

\newcommand{\la}{\langle}
\newcommand{\ra}{\rangle}
\newcommand{\flip}{F}
\newcommand{\Tr}{\mathrm{Tr}}
\newcommand{\ip}[2]{\langle #1,#2\rangle}
\newcommand{\bigip}[2]{\left\langle #1,#2\right\rangle}

\newcommand{\id}{\mathrm{Id}}

\newcommand{\ol}[1]{\overline{#1}}

{\begin{verse}%
		\small }%
	{\end{verse}}

\def\tilde{\widetilde}

\begin{document}
	\title{Word Trace Formulas over Compact Groups and Connected Cumulants}
	
\author{Liviu P\u aunescu}
\address{Institute of Mathematics of the Romanian Academy, 21 Calea Grivitei Street, 010702 Bucharest, Romania}
\email{liviu.paunescu@imar.ro}	

\author{Florin R\u adulescu}
\address{Department of Mathematics, University of Tor Vergata, Via della Ricerca Scientifica 1, 00133, Roma, Italy\\Institute of Mathematics of the Romanian Academy, 21 Calea Grivitei Street, 010702 Bucharest, Romania}
\email{radulesc@mat.uniroma2.it}
	
\begin{abstract}
We compute the expected value of the product of traces for a finite set of words within any compact subgroup of $\unit(N)$. Furthermore, we establish a sufficient condition under which these values determine that the product of the conjugacy classes of the words contains the identity.
\end{abstract}

	\maketitle

\section{Introduction}

Let $W=(w_1,w_2,\ldots,w_k)$ be a $k$-tuple of words on two letters. Define the function $I_W:\nz^{2k+1}\to\cz$ by:
\[I_W(N,i_1,\ldots,i_k,j_1,\ldots,j_k)=\int_{\unit(N)^2}\Pi_{c=1}^kTr(w_c(u,v))^{i_c}Tr(w_c^{-1}(u,v))^{j_c}dudv.\]

The integration is done with respect to the Haar measure on $\unit(N)$. We investigate how much information is contained in the function $I_W$. First of all, this function contains information on hyperlinearity.

\begin{te}
Let $\Gamma=\ff_2/\langle\!\langle w_1,\ldots,w_{k-1}\rangle\!\rangle$. Then $I_W$ can decide if there exists a morphism $\theta:\Gamma\to \unit(R^\omega)$ such that $\theta(w_k)\neq Id$.
\end{te}
\begin{proof}
We want to find $(u,v)\in\unit(N)^2$ such that $tr(w_i(u,v))$ is close to $1$ for $i=1,\ldots,k-1$ and $tr(w_k(u,v))$ is close to $0$.

For $(X,\mu)$ a finite space with measure and $f:X\to\rz_+$ we have:
\[max\{f(x):x\in X\}=\lim_{n\to\infty}\sqrt[n]{\int_Xf^n(y)d\mu(y)}.\]

Apply this statement for $X=\unit(N)^2$ endowed with the Haar measure and \[f(u,v)=max\{0,\sum_{i=1}^{k-1}Re(tr(w_i(u,v)))-Re(tr(w_k(u,v)))\}.\] 
The function $f$ can be computed from $I_W$, as $max\{a,b\}=(a+b+\sqrt{(a+b)^2})/2$ and $(1-tr(w_i))(1-tr(w_i)^*)=2-2Re(tr(w_i))$.
\end{proof}

If every group were hyperlinear, then the existance of such a $\theta$ is equivalent to $w_k\notin \langle\!\langle w_1,\ldots,w_{k-1}\rangle\!\rangle$. As such we find the following question natural.

\begin{qu}\label{q:normal group}
Can the function $I_W$ decide if $w_k\in \langle\!\langle w_1,\ldots,w_{k-1}\rangle\!\rangle$?
\end{qu}

Related to this:

\begin{qu}
Assume there exists $n\in\nz$ such that $|1-Tr(w_k(u,v))| \leqslant n\cdot\sum_{c=1}^{k-1}|1-Tr(w_c(u,v)|$, for each $N$, each $u,v\in\unit(N)$. Does it follow that $w_k\in \langle\!\langle w_1,\ldots,w_{k-1}\rangle\!\rangle$?
\end{qu}

\subsection{Main results and organisation of the paper}

In this paper we provide a general formula for the expected value of $\Pi_{c=1}^kTr(w_c(u,v))$, when $u,v\in G$, where $G$ is any compact subgroup of $\unit(N)$.

\begin{te}(Theorem \ref{te:main result})
	Let $G$ be a compact group, $\mu$ be its Haar measure. Then:
	\[\int_{G^2}\Pi_{c=1}^kTr(w_c(g,h))d\mu(g)d\mu(h)=Tr_{B(H^{\otimes r})}\big( (\Omega_{m_u}\otimes\Omega_{m_v})P_\delta P_\gamma\big).\]
\end{te}

Here $\Omega_m$ is the Casimir element, see Definition \ref{de:Casimir element} and $\gamma\in\s(r)$ is a permutation associated to the family of words $W$, constructed in the next subsection.

In the second part of the article we use connected cummulands to give a partial answer to Question \ref{q:normal group}.

\begin{te}(Theorem \ref{te:leading coefficient})
	If $c_0^W\neq 0$, then the family of words $W$ is planar.
\end{te}

The paper is organised as follows. In Section 2 we establish some preliminary algebraic computations. Section 3 is dedicated to the Casimir element, that appears in the final general formula of $I_W$. In Sections 4 and 5 we discuss the unitary case, in the stable and non-stable case respectively. Section 6 contains an argument that the function $I_W$ might actually contain little information on the family of words $W$. In Section 7 we discuss paths between bijections, to be used in the cummuland formula provided in Section 8. Section 9 proves the Euler estimate for a path and in Section 10 we prove the main result about planarity.

\subsection*{Acknowledgements}

Results belong to the authors. ChatGpt was used in the early stages of writing of this paper to assists with proofs, especially with the translation of toplogical statements in sections 8-10. The final form of the paper was completely transformed and rewritten.

\subsection{Balanced words in two variables}\label{balanced words}

Let $w_1,\dots,w_k$ be words in the alphabet $\{u,u^*,v,v^*\}$. We assume each word family is \emph{balanced} in the sense that the total number of occurrences of $u$ equals that of $u^*$, and similarly for $v$ and $v^*$.

Let
\[ m_u=\#(u)=\#(u^*),\qquad m_v=\#(v)=\#(v^*).\]
Thus the total number of letters is
\[ r=2m_u+2m_v.\]

Enumerate the letters occurring in the family $(w_1,\dots,w_k)$ from left to right. Let $f:\{1,\ldots,r\}\to\{u,u^*,v,v^*\}$ be the spelling of their product such that:
\begin{align*}
	f^{-1}(u)=&\{i_1,\ldots,i_{m_u}\},\mbox{ with } i_1<i_2<\ldots<i_{m_u};\\
	f^{-1}(u^*)=&\{j_1,\ldots,j_{m_u}\},\mbox{ with } j_1<j_2<\ldots<j_{m_u};\\
	f^{-1}(v)=&\{r_1,\ldots,r_{m_v}\},\mbox{ with } r_1<r_2<\ldots<r_{m_v};\\
	f^{-1}(v^*)=&\{s_1,\ldots,s_{m_v}\},\mbox{ with } s_1<s_2<\ldots<s_{m_v}.
\end{align*}

Construct $\sigma\in Sym(r)$ having the cycle decomposition corresponding to the lengths of the $w_1,\ldots,w_k$ words:
\[\sigma=\big(1,2,\ldots,lg(w_1)\big)\big(lg(w_1)+1,\ldots,lg(w_1)+lg(w_2)\big)\ldots\big(lg(w_1)+\ldots+lg(w_{k-1})+1,\ldots,r\big).\]
For example, if $m_u=2$, $m_v=3$, $w_1=uvu^{-1}v^{-1}$ and $w_2=uv^2u^{-1}v^{-2}$ then $\sigma=(1,2,3,4)(5,6,7,8,9,10)$.

We want a permutation that encodes the words $w_1,\ldots,w_k$. Construct $\lambda\in Sym(r)$ by:
\[\lambda(t)=i_t;\lambda(m_u+t)=j_t;\lambda(2m_u+t)=r_t;\lambda(2m_u+m_v+t)=s_t.\]

Define $\gamma=\lambda^{-1}\sigma\lambda$. Cycles in $\gamma$ corresponds to words $w_1,\ldots,w_k$. An $\{1,\ldots,m_u\}$ entry corresponds to letter $u$, a $\{m_u+1,\ldots,2m_u\}$ entry corresponds to letter $u^*$ and so on. For the above example, $w_1=uvu^{-1}v^{-1}$ and $w_2=uv^2u^{-1}v^{-2}$, we have $\gamma=(1, 5, 3, 8)(2,6,7,4,9,10)$. 

Let $G$ be a compact subgroup of $\unit(N)$ and denote by $\mu$ its Haar measure. Define:
\[I_\gamma^G=\int_{G^2}\Pi_{c=1}^kTr(w_c(u,v))d\mu(u)d\mu(v).\]

In other words $I_\gamma^G$ is the expected value of $\Pi_{c=1}^kTr(w_c(u,v))$ over $G^2$. If $G=\unit(N)$ we write $I_\gamma^N$ instead of $I_\gamma^{\unit(N)}$. We compute these moments and their connected cumulants. Theorem \ref{te:leading coefficient} shows that a nonzero leading connected coefficient forces a product-of-conjugates relation. This gives a necessary relation condition, without deciding the general problem.

\begin{ob}
If the famility of words is not balanced, it is easy to see that $I_\gamma^N=0$, by the properties of the Haar measure. However, for an albitrary compact subgroup $G$, unbalanced moments need not vanish. This happens for $G=\s(N)$, for eample.
\end{ob}

\section{Computations in $B(H^{\otimes r})$}
	
Throughout the paper, $H$ is a finite dimensional Hilbert space over $\cz$ of dimension $N\in\nz^*$. The \emph{inner product}, denoted by $\ip{\cdot}{\cdot}_H$ is linear in the first variable and conjugate-linear in the second.

Choose $\{e_i\}_{i=1}^N$ an ortonormal basis in $H$. \emph{The trace} of a bounded operator $T\in B(H)$ is defined as:
\[Tr_{B(H)}(T)=\sum_{i=1}^N\bigip{T(e_i)}{e_i}_H.\]

If $S$ is a finite group (usually the symmetric group), and $g\in S$, we also denote by $g\in L(S)$, the corresponding unitary in the group algebra. We denote by $\tau$ the group algebraic trace: $\tau(\sum_{g\in S}\lambda_gg)=\lambda_{Id}$.

\subsection{Permutation encoding of trace monomials}\label{sec:words}

The first step in computing the expected value of $\Pi_{c=1}^kTr(w_c(u,v))$ is to obtain this product as a trace of an operator on a bigger Hilbert space.

\subsection{Permutation trace identity}

For $\gamma\in \s(r)$, define $P_\gamma\in B(H^{\otimes r})$ by:
\[P_\gamma(\xi_1\otimes \cdots \otimes \xi_r)=\xi_{\gamma(1)}\otimes \cdots \otimes \xi_{\gamma(r)}.\]

\begin{p}[\cite{Pro21}, Proposition 1.5]\label{prop:trace-perm}
Let $A_1,\dots,A_r\in B(H)$. Let $\gamma\in S_r$ with cycle decomposition
	\[\gamma=(i_1^{(1)}\,i_2^{(1)}\,\dots\,i_{\ell_1}^{(1)})\cdots(i_1^{(s)}\,i_2^{(s)}\,\dots\,i_{\ell_s}^{(s)}).\]
	Then
	\[\Tr_{B(H^{\otimes r})}\bigl(P_\gamma(A_1\otimes \cdots \otimes A_r)\bigr)=\prod_{c=1}^s \Tr_{B(H)}\bigl(A_{i_1^{(c)}}A_{i_2^{(c)}}\cdots A_{i_{\ell_c}^{(c)}}\bigr).\]
\end{p}

\begin{proof}
	Fix an orthonormal basis $(e_1,\dots,e_N)$ of $H$. Then:
\begin{align*}
	\Tr_{B(H^{\otimes r})}\bigl(P_\gamma(A_1\otimes \cdots \otimes A_r)\bigr)=&\sum_{a_1,\dots,a_r}\Bigl\langle P_\gamma(A_1\otimes \cdots\otimes A_r)(e_{a_1}\otimes \cdots \otimes e_{a_r}),e_{a_1}\otimes \cdots \otimes e_{a_r}\Bigr\rangle\\
	=&\sum_{a_1,\dots,a_r}\Bigl\langle P_\gamma(A_1e_{a_1}\otimes \cdots\otimes A_re_{a_r}),e_{a_1}\otimes \cdots \otimes e_{a_r}\Bigr\rangle\\
	=&\sum_{a_1,\dots,a_r}\Bigl\langle A_{\gamma(1)}e_{a_{\gamma(1)}}\otimes \cdots\otimes A_{\gamma(r)}e_{a_{\gamma(r)}},e_{a_1}\otimes \cdots \otimes e_{a_r}\Bigr\rangle\\
	=&\sum_{a_1,\dots,a_r}\Pi_{p=1}^r \ip{A_{\gamma(p)}e_{a_{\gamma(p)}}}{e_{a_p}}
\end{align*}
For the right hand side:
\begin{align*}
\Tr_{B(H)}\bigl(A_{i_1^{(c)}}A_{i_2^{(c)}}\cdots A_{i_{\ell_c}^{(c)}}\bigr)=&\sum_{a_{i_1^{(c)}}a_{i_2^{(c)}}\cdots a_{i_{\ell_c}^{(c)}}}A_{i_1^{(c)}}(a_{i_1^{(c)}},a_{i_2^{(c)}})\cdot A_{i_2^{(c)}}(a_{i_2^{(c)}},a_{i_3^{(c)}})\cdot\ldots\cdot A_{i_{\ell_c}^{(c)}}(a_{i_{\ell_c}^{(c)}},a_{i_1^{(c)}})\\
=&\sum_{a_{i_1^{(c)}}a_{i_2^{(c)}}\cdots a_{i_{\ell_c}^{(c)}}}A_{i_1^{(c)}}(a_{i_1^{(c)}},a_{\gamma(i_1^{(c)})})\cdot A_{i_2^{(c)}}(a_{i_2^{(c)}},a_{\gamma(i_2^{(c)})})\cdot\ldots\cdot A_{i_{\ell_c}^{(c)}}(a_{i_{\ell_c}^{(c)}},a_{\gamma(i_{\ell_c}^{(c)})})\\
=&\sum_{a_{i_1^{(c)}}a_{i_2^{(c)}}\cdots a_{i_{\ell_c}^{(c)}}}\Pi_{p=1}^{i_{\ell_c}^{(c)}}A_{i_p^{(c)}}(a_{i_p^{(c)}},a_{\gamma(i_p^{(c)})})=\sum_{a_{i_1^{(c)}}a_{i_2^{(c)}}\cdots a_{i_{\ell_c}^{(c)}}}\Pi_{p=1}^{i_{\ell_c}^{(c)}}\ip{A_{i_p^{(c)}}e_{a_{\gamma(i_p^{(c)})}}}{e_{a_{i_p^{(c)}}}}
\end{align*}	
As $\{i_1^{(1)},i_2^{(1)},\dots,i_{\ell_1}^{(1)},\ldots\ldots, i_1^{(s)},i_2^{(s)},\dots,i_{\ell_s}^{(s)}\}=\{1,\ldots,r\}$, the formula now follows.
\end{proof}

 We are now ready for the main theorem of this section. Recall that $w_1,\ldots,w_k$ is a colection of balanced words, with total number of letters $r=2m_u+2m_v$, and $\gamma\in\s(r)$ is the associated permutation. 

\begin{te}
For any $u,v\in\unit(N)$, we have: 
	\[\Pi_{c=1}^kTr_{B(H)}(w_c(u,v))=Tr_{B(H^{\otimes r})}\big(P_\gamma(u^{\otimes m_u}\otimes(u^*)^{\otimes m_u}\otimes v^{\otimes m_v}\otimes(v^*)^{\otimes m_v})\big).\]
\end{te}
\begin{proof}
	This is a consequence of Proposition \ref{prop:trace-perm}, and of the construction of $\gamma\in\s(r)$.
\end{proof}

\begin{cor}\label{initial cor}
Let $G$ be a compact subgroup of $\unit(H)\subset B(H)$ and $\mu$ be its Haar measure. Then:
\[\int_{G^2}\Pi_{c=1}^kTr_{B(H)}(w_c(u,v))d\mu(u)d\mu(v)=Tr_{B(H^{\otimes r})}\big(P_\gamma(\int_Gu^{\otimes m_u}\otimes(u^*)^{\otimes m_u}d\mu(u)\otimes\int_G v^{\otimes m_v}\otimes(v^*)^{\otimes m_v}d\mu(v))\big).\]
\end{cor}

This Corollary reduces the computation of the expected value of $\Pi_{c=1}^kTr(w_c(u,v))$ to the computation of the operator $\int_Gu^{\otimes m_u}\otimes(u^*)^{\otimes m_u}$. Note that this operator does not depend on the words $w_1,\ldots, w_k$. That information is stored in the permutation $\gamma$.

\subsection{Integral of Ad}

The first step determing the value of $\int_Gu\otimes u^*d\mu(u)$ is to compute $\int_GAd(u)d\mu(u)$. We later see how to get from $Ad(u)$ to $u\otimes u^*$. We now consider $G$ a compact group, $\mu$ its Haar measure as before, and $\rho:G\to B(H)$ a continuous representation. Let $A_\rho=\rho(G)^\prime=\{T\in B(H):T\rho(g)=\rho(g)T,\ \forall g\in G\}$ and $E_\rho\in B(B(H))$ be the projection on $A_\rho$. The next Theorem is folklore.

\begin{te}\label{int of ad}
	We have: $\int_G Ad(\rho(g))d\mu(g)=E_\rho$.
\end{te}
\begin{proof}
	Let $F=\int_G Ad(\rho(g))d\mu(g)$, so $F(T)=\int_G \rho(g)T\rho(g)^*d\mu(g)$, for $T\in B(H)$. We prove four facts:
	
	\emph{1. $F(T)\in A_\rho$ for all $T\in B(H)$.} Fix $h\in G$. Then, using the invariance of the Haar measure:
	\[\rho(h)F(T)\rho(h)^*=\int_G\rho(hg)T\rho((hg))^*d\mu(g)=\int_G \rho(g)T\rho(g)^*d\mu(g)=F(T).\]
	
	\emph{2. $F(T)=T$ for all $T\in A_\rho$.} As $\rho(g)T\rho(g)^*=T$ for all $g\in G$ we get:
	\[F(T)=\int_GTd\mu(g)=T.\]
	
	\emph{3. $F$ is self adjoint.} Let $S,T\in B(H)$. Then:
	\[\ip{F(S)}{T}_{B(H)}=\int_GTr_H(T^*\rho(g)S\rho(g)^*)=\int_GTr_H\big((\rho(g)^*T\rho(g))^*S\big)=\ip{S}{F(T)}.\]
	
	\emph{4. $F^2=F$.} By 1. $F(T)\in A_\rho$, and by 2. $F(F(T))=F(T)$. So $F^2=F$.
	
	Properties 3. and  4. show that $F\in B(B(H))$ is a projection, and properties 1. and 2. show that the range of $F$ is $A_\rho$. As such $F=E_\rho$.
\end{proof}

\subsection{Operators on $B(H)$}

We want to compare $u\otimes u^*$, an element of $B(H\otimes H)$, to $Ad(u)$, an element in $B(B(H))$. We recall that $B(H)\simeq H\otimes\ol H$, an isomorphism we now describe.

\begin{de}
	The \emph{conjugate Hilbert space} $\ol H$ is the vector space consisting of formal symbols $\ol \xi$, $\xi\in H$, with addition
	\[\ol \xi+\ol \eta:=\ol{(\xi+\eta)},\]
	scalar multiplication
	\[	\lambda \cdot \ol \xi:=\ol{\ol\lambda\,\xi},\]
	and inner product
	\[\ip{\ol\xi}{\ol\eta}_{\ol H}:=\ip{\eta}{\xi}_H.\]
\end{de}

\begin{nt}
For $\xi,\eta\in H$, we denote by $\xi\ip{\cdot}{\eta}_H\in B(H)$ the rank one operator:
\[z\mapsto \langle z,\eta\rangle_H\cdot\xi.\]
\end{nt}

The following Proposition is folclore.

\begin{p}
	There is a canonical Hilbert space isomorphism:
	\[\Theta:H\otimes \ol H\longrightarrow B(H)\]
	characterized by
	\[\Theta(\xi\otimes \ol\eta)=\xi\langle\,\cdot\,,\eta\rangle_H.\]
\end{p}

\begin{ob}
In this isomorphism $B(H)$ is considered a Hilbert space with the inner product $\ip{T}{S}_{B(H)}=Tr_{B(H)}(S^*T)$.
\end{ob}

\begin{p}
We have $\big(\xi\langle\,\cdot\,,\eta\rangle_H\big)^*=\eta\langle\,\cdot\,\xi\rangle_H$.
\end{p}

\begin{p}\label{Calcul1}
For $a,b\in H$ and $T\in B(H)$ we have $\left\la a\la\cdot,b\ra_H,T\right\ra_{B(H)}=\la a,Tb\ra_H$.
\end{p}
\begin{proof}
Fix $\{e_i\}_{i=1}^N$ a basis in $H$. If $T=e_i\la\cdot,e_j\ra_H$, then:
\[\la a\la\cdot,b\ra_H,T\ra_{B(H)}=\la a\otimes\ol b,e_i\otimes\ol e_j\ra_{H\otimes\ol H}=\la a,e_i\ra_H\cdot\la e_j,b\ra_H=\left\la\la a,e_i\ra_He_j,b\right\ra_H=\la T^*a,b\ra_H=\la a,Tb\ra_H.\]
The result follows by liniarity.
\end{proof}

\begin{de} 
	Define $\Psi:B(H\otimes \ol H)\to B(B(H))$, by the rule $\Psi(X)=\Theta\circ X\circ \Theta^{-1}$. That is, the following diagram is commutative:
	\[\begin{array}{ccc}
		H\otimes \ol H & \xrightarrow{\ X\ } & H\otimes \ol H \\
		\Theta\downarrow \phantom{\Theta} & & \phantom{\Theta}\downarrow \Theta \\
		B(H) & \xrightarrow{\ \Psi(X)\ } & B(H)
	\end{array}\]
\end{de}

\begin{de}
For $A\in B(H)$ define $L_A,R_A\in B(B(H))$ by $L_A(T)=AT$ and $R_A(T)=TA$.
\end{de}

\begin{p}\label{psi formula}
For $A,B\in B(H)$ we have $\Psi(A\otimes\ol B)=L_AR_{B^*}$.
\end{p}
\begin{proof}
It sufices to verifiy the idenity on rank 1 operators. For $\xi,\eta\in  H$ we have:
\[\Psi(A\otimes\ol B)(\xi\ip{\cdot}{\eta}_H)=\Theta\circ(A\otimes\ol B)(\xi\otimes\ol\eta)=\Theta(A\xi\otimes\ol{B\eta})=A\xi\ip{\cdot}{B\eta}_H.\]
On the other hand, for $a\in H$:
\[\big(L_AR_{B^*}(\xi\ip{\cdot}{\eta}_H)\big)(a)=(A(\xi\ip{\cdot}{\eta}_H)B^*)(a)=\big(A(\xi)\ip{\cdot}{\eta}_H\big)B^*(a)=A(\xi)\ip{B^*(a)}{\eta}_H)=(A\xi\ip{\cdot}{B\eta}_H)(a).\]
As such $\Psi(A\otimes\ol B)(\xi\ip{\cdot}{\eta})=L_AR_{B^*}(\xi\ip{\cdot}{\eta})$.
\end{proof}	

As a corollary, $Ad(u)\in B(B(H))$ is sent, via $\Psi^{-1}$, to $u\otimes\ol u\in B(H\otimes\ol H)$. Before passing from $u\otimes\ol u$ to $u\otimes u^*$, we study the behaviour of $\Psi^{-1}$ on rank 1 operators, which is done in Theorem \ref{psi on rank 1}.

\subsection{Preimage of rank 1 operators in $B(B(H))$.} The first step is to find the element $\chi\in H\otimes\ol H$, such that $\Theta(\chi)=Id_H\in B(H)$. Recall that $\{e_i\}_{i=1}^N$ is a basis in $H$.

\begin{de}
	Define $\chi=\sum_{i=1}^Ne_i\otimes\ol e_i\in H\otimes\ol H$.
\end{de}	

\begin{p}
	We have: $\Theta(\chi)=\id_H$.
\end{p}	
\begin{proof}
	Trivial, as $\Theta(\chi)=\sum_{i=1}^Ne_i\langle\cdot,e_i\rangle=\id_H$.
\end{proof}

\begin{p}
For $a,b\in H$, we have $\langle a\otimes\ol b,\chi\rangle_{H\otimes\ol H}=\langle a,b\rangle_H$.
\end{p}
\begin{proof}
Let $e_a,e_b$ be elements in the chosen basis. Then:
\[\la e_a\otimes\ol e_b,\chi\ra_{H\otimes\ol H}=\sum_j\la e_a\otimes\ol e_b,e_j\otimes\ol e_j\ra_{H\otimes\ol H}=\sum_j\delta_a^j\delta_b^j=\delta_a^b=\la e_a,e_b\ra_H.\]
The result follows by linear extension.
\end{proof}

We are now ready to compute $\Psi^{-1}(S\la\cdot,T\ra_{B(H)})$, which is the goal of this subsection. As a notation, for $T\in B(H)$, denote by $\ol T\in B(\ol H)$, the operator: $\ol\xi\mapsto \ol{T\xi}$.

\begin{te}\label{psi on rank 1}
We have: $\Psi\big((S\otimes\ol\id)(\chi\la\cdot,\chi\ra_{H\otimes\ol H})(\id\otimes\ol T)\big)=S\la\cdot,T\ra_{B(H)}$.
\end{te}	
\begin{proof}
Let $X=(S\otimes\ol\id)(\chi\la\cdot,\chi\ra_{H\otimes\ol H})(\id\otimes\ol T)\in B(H\otimes\ol H)$. For $a,b\in H$:

\begin{align*}
X(a\otimes\ol b)=&(S\otimes\ol\id)(\chi\la\cdot,\chi\ra_{H\otimes\ol H})(a\otimes\ol{Tb})=(S\otimes\ol\id)(\la a,Tb\ra_H\cdot\chi) \\=&\la a,Tb\ra_H(S\otimes\ol\id)\sum_ie_i\otimes\ol e_i=\la a,Tb\ra_H\sum_iS(e_i)\otimes\ol e_i.
\end{align*}
If we apply $\Theta$, we get:
\[\Theta\big(\la a,Tb\ra_H\sum_iS(e_i)\otimes\ol e_i\big)=\la a,Tb\ra_H\sum_iS(e_i)\la\cdot,e_i\ra_H=\la a,Tb\ra_HS(\cdot).\]
On the other hand, by Proposition \ref{Calcul1}, $S\la\cdot,T\ra_{B(H)}(a\la\cdot,b\ra_H)=\la a,Tb\ra_HS(\cdot)$. As such, for each $a,b\in H$:
\[\Theta\circ X\circ\Theta^{-1}(a\la\cdot,b\ra_H)=S\la\cdot,T\ra_{B(H)}(a\la\cdot,b\ra_H).\]
It follows that $\Psi(X)=S\la\cdot,T\ra_{B(H)}$.
\end{proof}

\subsection{The transpose} This is the map that sends $\ol U$ to $U^*$. This subsection will also explain why the \emph{flip operator} appears in the final formula. 

\begin{de}
Define $\tc:B(\ol H)\to B(H)$ by $\tc(\ol U)=U^*$. Note that $\tc(\ol U\cdot\ol V)=V^*U^*$.
\end{de}	

So $(Id\otimes\tc)\Psi^{-1}(Ad(u))=u\otimes u^*$. This is why $(Id\otimes\tc)\Psi^{-1}:B(B(H))\to B(H\otimes H)$ is a key operator for us, that we now investigate.

\begin{ex}\label{ex:transpose}
	For $\xi,\eta\in H$, $\tc(\ol\xi\langle\cdot,\ol\eta\rangle_{\ol H})=\eta\langle\cdot,\xi\rangle_H$.
\end{ex}

\begin{p}\label{Delta}
	We have $\chi\ip{\cdot}{\chi}_{H\otimes\ol H}=\sum_{i,j}e_i\langle\cdot,e_j\rangle_H\otimes\ol e_i\langle\cdot,\ol e_j\rangle_{\ol H}$.
\end{p}
\begin{proof}
For $e_a,e_b$ in the chosen basis we have:
\begin{align*}
\big(\sum_{i,j}e_i\langle\cdot,e_j\rangle_H\otimes\ol e_i\langle\cdot,\ol e_j\rangle_{\ol H}\big)(e_a\otimes\ol e_b)=&\sum_{i,j}e_i\langle e_a,e_j\rangle_H\otimes\ol e_i\langle\ol e_b,\ol e_j\rangle_{\ol H}=\sum_{i,j}\delta_a^je_i\otimes\delta_b^j\ol e_i=\delta_a^b\sum_ie_i\otimes\ol e_i\\
=&\la e_a,e_b\ra\chi=\big(\chi\ip{\cdot}{\chi}_{H\otimes\ol H}\big)(e_a\otimes\ol e_b).
\end{align*}
\end{proof}

\begin{de}
Let $\flip\in B(H\otimes H)$ be the \emph{flip operator}: $\flip(x\otimes y)=y\otimes x$.
\end{de}

\begin{ob}\label{ob:flip operator}
The flip operator $F$ is equal to $P_{(12)}$, where $(12)\in\s(2)$ is the transposition.
\end{ob}

\begin{p}
	We have $\id\otimes\tc(\chi\ip{\cdot}{\chi}_{H\otimes\ol H})=\flip$.
\end{p}
\begin{proof}
By \ref{ex:transpose} and \ref{Delta} , $\id\otimes\tc(\chi\ip{\cdot}{\chi}_{H\otimes\ol H})=\sum_{i,j}e_i\langle\cdot,e_j\rangle_{H}\otimes e_j\langle\cdot, e_i\rangle_{H}$. Moreover:
\[\big(\sum_{i,j}e_i\langle\cdot,e_j\rangle_{H}\otimes e_j\langle\cdot, e_i\rangle_{H}\big)(e_a\otimes e_b)=\sum_{i,j}e_i\delta_a^j\otimes e_j\delta_b^i=e_b\otimes e_a.\]
This proves the conclusion.	
\end{proof}

We are now ready for the main result of this section.

\begin{te}\label{psi^-1}
	We have $(\id\otimes\tc)\big(\Psi^{-1}(S\langle\,\cdot\,,T\rangle_{B(H)})\big)=(S\otimes T^*)\flip$.
\end{te}	
\begin{proof}
By \ref{psi on rank 1}, $\Psi^{-1}(S\langle\,\cdot\,,T\rangle_{B(H)})=(S\otimes\ol\id)\chi\ip{\cdot}{\chi}_{H\otimes\ol H}(\id\otimes\ol T)$. As $\tc(\ol U\cdot \ol V)=\tc(\ol V)\tc(\ol U)$, we get:
\[(\id\otimes\tc)\big(\Psi^{-1}(S\langle\,\cdot\,,T\rangle_{B(H)})\big)=(S\otimes\id)(\id\otimes T^*)\big((\id\otimes\tc)\chi\ip{\cdot}{\chi}_{H\otimes\ol H}\big)=(S\otimes T^*)\flip.\]
\end{proof}

\subsection{Revisiting the integral}

We can do a first computation of $\int_Gu\otimes u^*d\mu(u)$. Recall that $\rho:G\to B(H)$ is a representation, and $E_\rho\in B(B(H))$ is the projection on $A_\rho=(\rho(G))^\prime$.

\begin{te}\label{integral gog^*}
	We have $\int_G\rho(g)\otimes\rho(g)^*d\mu(g)=(\id\otimes\tc)(\Psi^{-1}(E_\rho))$.
\end{te}
\begin{proof}
	By Proposition \ref{psi formula}, $\Psi(A\otimes\ol A)=L_AR_{A^*}=Ad(A)$. As such:
	\[A\otimes A^*=(\id\otimes\tc)(A\otimes \ol A)=(\id\otimes\tc)\Psi^{-1}(Ad(A)).\]
	Integrating over $G$, and by Theorem \ref{int of ad}:
	\[\int_G\rho(g)\otimes\rho(g)^*d\mu(g)=(\id\otimes\tc)\big(\Psi^{-1}(\int_GAd(\rho(g))d\mu(g))\big)=(\id\otimes\tc)(\Psi^{-1}(E_\rho)).\]
\end{proof}

By Theorem \ref{psi^-1}, we know how to compute $(\id\otimes\tc)\Psi^{-1}$ on rank one operators. As such the goal now is to write $E_\rho$ as a sum of rank one operators. This is done in the next section, via the Casimir element.

\section{Casimir Element}

Let $(A,Tr)$ be a finite dimensional von Neumann algebra endowed with a faitful finite trace. In this section we construct a canonical element $\Omega_A$ in $A\otimes A$ such that for any $x,y\in A$:
\[Tr_{A\otimes A}\big((x\otimes y)\Omega_A\big)=Tr_A(xy).\]

\begin{ex}
Let $A=M_n(\cz)$, $\{e_{ij}\}_{i,j=1}^n$ be its matrix units. Consider the usual trace: $Tr_A(e_{11})=1$. Then $\Omega_A=\sum_{i,j}e_{ij}\otimes e_{ji}$. Indeed:
\[Tr_{A\otimes A}\big((e_{ab}\otimes e_{cd})\Omega_A\big)=Tr_{A\otimes A}\big((e_{ab}\otimes e_{cd})(e_{bd}\otimes e_{db})\big)=Tr_{A\otimes A}(e_{ad}\otimes e_{cb})=\delta_a^d\delta_c^b=Tr_A(e_{ab}e_{cd}).\]
\end{ex}

The general case is not more complicated, as any finite dimensinal von Neumann algebra is a direct sum of matrix algebras. We present the general construction using dual basis.

\begin{de}
Let $(c_i)_i$ be a base of $A$. The \emph{dual base} of $(c_i)_i$ is a base $(u_i)_i$ of $A$ such that $Tr_A(u_j^*c_i)=\delta_i^j$.
\end{de}

\begin{p}
Let $(c_i)_{i=1}^m$ be a base of $A$. Let $G=(g_{ij})_{ij}\in M_m(\cz)$ defined by $g_{ij}=Tr_A(c_j^*c_i)$. Then $G$ is invertible. If $W=G^{-1}$, $W=(w_{ij})_{ij}$ and $u_i=\sum_jw_{ij}c_j$, then $(u_i)_i$ is the dual base of $(c_i)_i$.
\end{p}
\begin{proof}
Assume that there exists $(\lambda_i)_i\subset\cz$, such that $\sum_i\lambda_ig_{ij}=0$ for all $j=1,\ldots,m$. Then $\sum_i\lambda_iTr_A(c_j^*c_i)=Tr_A(c_j^*\sum_i\lambda_ic_i)=0$. As $(c_j^*)_j$ is also a basis in $A$, it follows that $\sum_i\lambda_ic_i=0$. As such $\lambda_i=0$ for all $i$. For the second part \[Tr(u_j^*c_i)=\sum_{k=1}^mTr(\overline{w_{jk}}c_k^*c_i)=\sum_{k=1}^m \overline{w_{jk}Tr(c_i^*c_k)} =\sum_{k=1}^m\overline{w_{jk}g_{ki}}=\overline{\delta_i^j}=\delta_i^j.\]
\end{proof}

\begin{de}\label{de:Casimir element}
	We define the \emph{Casimir element} of $A$ by the formula: 
	\[\Omega_A=\sum_{i=1}^mc_i\otimes u_i^*\in A\otimes A.\]
\end{de}

\begin{p}\label{trace of Frobenius}
	For any $x,y\in A$, $Tr_{A\otimes A}\big((x\otimes y)\Omega_A\big)=Tr_A(xy)$.
\end{p}
\begin{proof}
	By linearity, it is enough to prove this equation for $x=(u_a)^*$ and $y=c_b$.	Then:
	\[Tr\big((u_a^*\otimes c_b)\Omega_A\big)=\sum_{i=1}^mTr\big(u_a^*c_i\otimes c_bu_i^*\big)= \sum_{i=1}^mTr(u_a^*c_i) Tr(u_i^*c_b)= \sum_{i=1}^m\delta_a^i\delta_i^b= \delta_a^b=Tr(u_a^*c_b).\]
\end{proof}

\begin{cor}
	The Casimir element $\Omega_A$ is independent of the set $\{c_i\}_{i=1}^m$.
\end{cor}
\begin{proof}
As $Tr_A$ is faithful, so is $Tr_{A\otimes A}$. Then $Tr_{A\otimes A}(xa)=Tr_{A\otimes A}(xb)$ for any $x$ implies $a=b$.
\end{proof}

\begin{cor}
The Casimir element is selfadjoint. If $A\subset B(H)$ then the Casimir element commutes with the flip operator on $H\otimes H$.
\end{cor}
\begin{proof}
We have $\Omega_A^*=\sum_{i=1}^mc_i^*\otimes u_i$, which is the Casimir element associated to the base $\{c_i^*\}_i$. For the second part, notice that $\flip\Omega_A\flip=\sum_{i=1}^m u_i^*\otimes c_i$, which is the Casimir element associated to the base $\{u_i^*\}_i$.
\end{proof}

\begin{cor}
Let $A=\oplus_{\lambda\in Y}M_{d_\lambda}(\cz)$ be a finite dimensional von Neumann algebra, and let $\{e_{i,j}^\lambda\}_{i,j=1,\ldots,d_\lambda}\subset A$ be the matricial units of $M_{d_\lambda}(\cz)$. Then the Casimir element of $A$ is:
\[\Omega_A=\sum_{\lambda\in Y}\frac1{Tr_A(e_{11}^\lambda)}\sum_{i,j=1}^{d_\lambda}e_{ij}^\lambda\otimes e_{ji}^\lambda.\]
\end{cor}
\begin{proof}
The formula comes from the fact that $\{e_{i,j}^\lambda\}_{i,j=1,\ldots,d_\lambda}^{\lambda\in Y}$ is a basis in $A$, and $\{1/{Tr_A(e_{11}^\lambda)}\cdot e_{i,j}^\lambda\}_{i,j=1,\ldots,d_\lambda}^{\lambda\in Y}$ is its dual basis.
\end{proof}

\begin{cor}
The central support of $\Omega_A$ is $\sum_{\lambda\in Y}p_\lambda\otimes p_\lambda$, where $\{p_\lambda\}_\lambda$ are the minimal central projections of $A$.
\end{cor}

\begin{ex}
Let $S$ be a finite group, and consider its group algebra $\big(L(S),\tau\big)$. Then $\{g\}_{g\in S}$ is a basis in $L(S)$, and the matrix $G=(\tau(s_j^{-1}s_i))_{ij}$ is the identity. As such, the dual base is still $\{g\}_{g\in S}$. It follows that the Casimir element of $L(S)$ is $\sum_{g\in S}g\otimes g^{-1}$. We introduce a special notation for this case, for a later use.
\end{ex}

\begin{nt}
Define $\ve_m\in L\big(\s(2m)\big)$, by $\ve_m=\sum_{\sigma\in\s(m)}\sigma\otimes \sigma^{-1}$.
\end{nt}

\begin{p}\label{p:Casimir two traces}
Let $(A,Tr_A)$, $(B,Tr_B)$ be two algebras and $\alpha:A\to B$ be a surjective unital $*$-morphism. Then there exists $R_\alpha\in A$ a central element such that $\Omega_B=\alpha^{\otimes 2}(R_\alpha\otimes Id_A\cdot\Omega_A)$.
\end{p}
\begin{proof}
Let $\{p_\lambda\}_{\lambda\in Y}$ be the collection of minimal central projections in $A$. Then $A=\oplus_{\lambda\in Y}M_{d_\lambda}(\cz)$, where $d_\lambda\in\nz^*$. Let $\{e_{i,j}^\lambda\}_{i,j=1,\ldots,d_\lambda}\subset A$ be the matricial units of $M_{d_\lambda}(\cz)$. Then:
\[\Omega_A=\sum_{\lambda\in Y}\frac1{Tr_A(e_{11}^\lambda)}\sum_{i,j=1}^{d_\lambda}e_{ij}^\lambda\otimes e_{ji}^\lambda.\]
Let $Y_\alpha=\{\lambda:\alpha(p_\lambda)\neq 0\}$. Due to the surjectivity of $\alpha$, $\{\alpha(p_\lambda):\lambda\in Y_\alpha\}$ is the set of minimal central projections in $B$, and $B=\oplus_{\lambda\in Y_\alpha}M_{d_\lambda}(\cz)$. Then:
\[\Omega_B=\sum_{\lambda\in Y_\alpha}\frac1{Tr_B(\alpha(e_{11}^\lambda))}\sum_{i,j=1}^{d_\lambda}\alpha(e_{ij}^\lambda)\otimes\alpha(e_{ji}^\lambda).\]
Define the central element:
\[R_\alpha=\sum_{\lambda\in Y_\alpha}\frac{Tr_A(e_{11}^\lambda)}{Tr_B(\alpha(e_{11}^\lambda))}p_\lambda\in A.\]
It is easy to check that $\alpha^{\otimes 2}(R_\alpha\otimes Id_A\cdot\Omega_A)=\Omega_B$.
\end{proof}

\subsection{Integral of $g\otimes g^*$} Now we show how $E_\rho$, the projection on $A_\rho$, is sent to the Casimir element by the operator $(\id\otimes\tc)\Psi^{-1}$. Let $A\subset B(H)$ be a unital $*$-subalgebra. We consider $A$ endowed with the trace inhereted from $B(H)$. Let $(c_i)_i$ be a base of $A$, and $(u_i)_i$ be its dual base.

\begin{p}
	For $y\in A$, we have $y=\sum_{i=1}^mTr_{B(H)}\big(u_i^*y\big)\cdot c_i$.
\end{p}
\begin{proof}
This follows from the definition of the dual base. By linearity, we check it for $y=c_a$. Then:
	\[\sum_{i=1}^mTr\big(u_i^*c_a\big)\cdot c_i=\sum_{i=1}^m\delta_i^ac_i=c_a.\]
\end{proof}

\begin{p}
	The operator $E_A:B(H)\to A$ defined by $E_A(x)=\sum_{i=1}^mTr_H(u_i^*x)c_i$ is the ortohonal projection onto $A$.
\end{p}
\begin{proof}
By construction $E_A(x)\in A$. Then, by previous Propostion, $E_A(x)=x$ if $x\in A$. It remains to show that $E_A$ is self-adjoint.

By definition $E_A(x)=\sum_{i=1}^mTr_H(u_i^*x)c_i=\sum_{i,j=1}^m \ol w_{ij}Tr_H(c_j^*x)c_i$. So $E_A=\sum_{i,j=1}^m\ol w_{ij}c_i\ip{\cdot}{c_j}_H$. Then: 
\[(E_A)^*=\sum_{i,j=1}^m w_{ij}\big(c_i\ip{\cdot}{c_j}_H\big)^*=\sum_{i,j=1}^m \ol w_{ji}c_j\ip{\cdot}{c_i}_H=E_A.\]
\end{proof}

\begin{te}\label{integral of uou^*}
Let $\rho:G\to B(H)$ a representation and let $\Omega_\rho$ be the Casimir element of $A_\rho$. Then:
\[\int_G\rho(g)\otimes\rho^*(g)d\mu(g)=\Omega_\rho\flip.\]
\end{te}
\begin{proof}
By Theorem \ref{integral gog^*}, $\int_G\rho(g)\otimes\rho(g)^*d\mu(g)=(\id\otimes\tc)(\Psi^{-1}(E_\rho))$. By the previous Propostion $E_\rho=\sum_{i=1}^mc_i\ip{\cdot}{u_i}_{B(H)}$. As such, by Propostion \ref{psi^-1}:
\[(\id\otimes\tc)(\Psi^{-1}(E_\rho))=\sum_{i=1}^m\big(c_i\otimes u_i^*\big)\flip=\Omega_\rho\flip.\] 
\end{proof}

\subsection{The general formula}

For $\rho:G\to B(H)$ and $m\in\nz^*$, construct $\rho_m:G\to B(H^{\otimes m})$, $\rho_m(g)=\rho(g)^{\otimes m}$. Denote by $\Omega_m$ the Casimir element of $(\rho_m(G))^\prime$.

Let $w_1,w_2\ldots, w_k$ be a familiy of balanced words. Consider $m_u,m_v,r\in\nz$ and $\gamma\in\s(r)$ constructed in Section \ref{balanced words}. Define $\delta_m\in\s(2m)$ by $\delta_m(t)=t\pm m$, such that $P_{\delta_m}=\flip^{m}$, flip operator on $H^{\otimes m}\otimes H^{\otimes m}$, see Observation \ref{ob:flip operator}. Denote by $\delta\in\s(r)$ the permutation $\delta_{m_u}\oplus\delta_{m_v}$, so $P_\delta=\flip^{m_u}\otimes\flip^{m_v}$.

\begin{te}\label{te:main result}
Let $G$ be a compact group, $\mu$ be its Haar measure and $\rho:G\to B(H)$ a representation. Then:
\[\int_{G^2}\Pi_{c=1}^kTr_{B(H)}(w_c(\rho(g),\rho(h)))d\mu(g)d\mu(h)=Tr_{B(H^{\otimes r})}\big( (\Omega_{m_u}\otimes\Omega_{m_v})P_\delta P_\gamma\big).\]
\end{te}
\begin{proof}
This is a consequence of Corollary \ref{initial cor} and Theorem \ref{integral of uou^*}. 
\end{proof}

\section{The unitary case}\label{s:unitary case}

In the next two sections we discuss the unitary case. This section is dedicated to the \emph{stable case}, i.e. when $N\geqslant k$. Before computing the Casimir element in this setting, we present an independent computation of the expected value of $\Pi_cTr_{B(H)}(w_c(u,v))$. The arguments are more combinatoric and directly based on Weingarten calculus. Results in this subsection were proved by Rădulescu \cite{Rad06}. It provides the algebraic starting point for the corresponding matching formula used in Magee–Puder  \cite[Theorem 3.7]{MP15}\cite[Theorem 2.8]{MP19}. We include the index calculation to relate it explicitly to the Casimir and group-algebra formulas developed here.

\subsection{A direct formula} In this subsection, we work mainly inside $L(\s(m))$, the group algebra of $\s(m)$. Permutations are denoted by $p,q$, and as before, we use the same notation for elements in the group $\s(m)$, or in its group algebra.

\begin{de}
	Denote by $\Phi_m^N\in L(\s(m))$, the \emph{Formanek element}:
	\[\Phi_m^N=\sum_{p\in\s(m)}N^{\# cyc(p)}p.\] 
\end{de}

The Formanek element is invertible if and only if $N\geqslant m$, see \cite{For87,Col03,CS06,Pro21}. The Weingarten numbers are the coefficients of its inverse. For smaller positive integral $N$, we use the inverse on the represented support, see Section 5.

\begin{de}
The \emph{Weingarten numbers} $W_p^N\in\rz$ are the coeficients of the inverse:
\[(\Phi_m^N)^{-1}=\sum_{p\in\s(m)}W_p^Np.\] 
\end{de}

\begin{ob}
By construction $\sum_{p\in\s(m)}W_p^NN^{\# cyc(p)}=1$ and $\sum_{p\in\s(m)}W_p^NN^{\# cyc(pq)}=0$ for any $q\neq Id$.	
\end{ob}

\begin{de}
	For $a_1,\ldots,a_m,b_1,\ldots,b_m\in\nz$ define
	\[P(a_1,\ldots,a_m;b_1,\ldots,b_m)=\{p\in\s(m):a_{p(i)}=b_i,\mbox{ for } i=1,\ldots,m\}.\]
\end{de}

If $u\in\unit(N)$ is a matrix, and $a,b\in\{1,\ldots,N\}$, we denote by $u_{ab}$ the $(a,b)$-entry of $u$. The following is the unitary Weingarten entry-integration formula. See Samuel \cite{Sam80}, Collins \cite{Col03}, and Collins--\'Sniady
\cite{CS06}; historical background, including Weingarten's asymptotic work \cite{Wei78} and Xu's work \cite{Xu97}, is reviewed in \cite{CMN22}.

\begin{te}
	Let $a_1,\ldots,a_m,b_1,\ldots,b_m,c_1,\ldots,c_m,d_1,\ldots,d_m$ be indices in the set $\{1,\ldots,N\}$. Then:
	\[\int_{\unit(N)}u_{a_1b_1}\ldots u_{a_mb_m}u^*_{c_1d_1}\ldots u^*_{c_md_m}d\mu(u)=\sum_{\substack{p\in P(a_1,\ldots,a_m;d_1,\ldots, d_m);\\ q\in P(b_1,\ldots,b_m;c_1,\ldots, c_m)}}W^N_{pq^{-1}}.\]
\end{te}

Recall that for a set of balanced words $w_1,\ldots,w_k$ we want to compute:
\[I_\gamma^N=\int_{\unit(N)^2}\Pi_{c=1}^kTr_{B(H)}(w_c(u,v))d\mu(u)d\mu(v).\]

\begin{lemma}
We have: 
\[I_\gamma^N=\sum_{a_1,\ldots,a_r=1}^N\big(\sum_{\substack{p_u\in P(a_{i_1},\ldots,a_{i_n};a_{\gamma(j_1)},\ldots,a_{\gamma(j_n)});\\ q_u\in P(a_{\gamma(i_1)},\ldots,a_{\gamma(i_n)};a_{j_1},\ldots,a_{j_n})}}W^N_{p_uq_u^{-1}}
\cdot\sum_{\substack{p_v\in P(a_{r_1},\ldots,a_{r_m};a_{\gamma(s_1)},\ldots,a_{\gamma(s_m)});\\ q_v\in P(a_{\gamma(r_1)},\ldots,a_{\gamma(r_m)};a_{s_1},\ldots,a_{s_m})}}W^N_{p_vq_v^{-1}}\big).\]
\end{lemma}
\begin{proof}
The function $f:\{1,\ldots,r\}\to\{u,u^*,v,v^*\}$ is the spelling of the words $w_1,\ldots,w_k$ and $\sigma\in\s(r)$ has cycles coresponding to these words. Then:
\begin{align*}
	&I_\gamma^N=\int\int_{\unit(N)^2}\sum_{a_1,\ldots,a_r=1}^Nf(1)_{a_1a_{\sigma(1)}}f(2)_{a_2a_{\sigma(2)}}\ldots f(r)_{a_ra_{\sigma(r)}}dudv= \\ 
	&\int\int_{\unit(N)^2}\sum_{a_1,\ldots,a_r=1}^N(u_{a_{i_1}a_{\sigma(i_1)}}\ldots u_{a_{i_{m_u}}a_{\sigma(i_{m_u})}}u^*_{a_{j_1}a_{\sigma(j_1)}}\ldots u^*_{a_{j_{m_u}}a_{\sigma(j_{m_u})}})
	(v_{a_{r_1}a_{\sigma(r_1)}}\ldots v_{a_{r_{m_v}}a_{\sigma(r_{m_v})}}v^*_{a_{s_1}a_{\sigma(s_1)}}\ldots v^*_{a_{s_{m_v}}a_{\sigma(s_{m_v})}})dudv=\\
	&\sum_{a_1,\ldots,a_r=1}^N\int_{\unit(N)}(u_{a_{i_1}a_{\sigma(i_1)}}\ldots u_{a_{i_{m_u}}a_{\sigma(i_{m_u})}}u^*_{a_{j_1}a_{\sigma(j_1)}}\ldots u^*_{a_{j_{m_u}}a_{\sigma(j_{m_u})}})du
	\int_{\unit(N)}(v_{a_{r_1}a_{\gamma(r_1)}}\ldots v_{a_{r_m}a_{\sigma(r_{m_v})}}v^*_{a_{s_1}a_{\sigma(s_1)}}\ldots v^*_{a_{s_{m_v}}a_{\sigma(s_{m_v})}})dv\\
	&=\sum_{a_1,\ldots,a_r=1}^N\big(\sum_{\substack{p_u\in P(a_{i_1},\ldots,a_{i_{m_u}};a_{\sigma(j_1)},\ldots,a_{\sigma(j_{m_u})});\\ q_u\in P(a_{\sigma(i_1)},\ldots,a_{\sigma(i_{m_u})};a_{j_1},\ldots,a_{j_{m_u}})}}W^N_{p_uq_u^{-1}}
	\cdot\sum_{\substack{p_v\in P(a_{r_1},\ldots,a_{r_{m_v}};a_{\sigma(s_1)},\ldots,a_{\sigma(s_{m_v})});\\ q_v\in P(a_{\sigma(r_1)},\ldots,a_{\sigma(r_{m_v})};a_{s_1},\ldots,a_{s_{m_v}})}}W^N_{p_vq_v^{-1}}\big)
\end{align*}
\end{proof}

Fix $p_u,q_u\in Sym(m_u)$ and $p_v,q_v\in Sym(m_v)$. We count how many tuples $(a_1,\ldots,a_r)\in\{1,\ldots,N\}^r$ have the property that $p_u\in P(a_{i_1},\ldots,a_{i_{m_u}};a_{\sigma(j_1)},\ldots,a_{\sigma(j_{m_u})})$; $q_u\in P(a_{\sigma(i_1)},\ldots,a_{\sigma(i_{m_u})};a_{j_1},\ldots,a_{j_{m_u}})$ and $p_v\in P(a_{r_1},\ldots,a_{r_{m_v}};a_{\sigma(s_1)},\ldots,a_{\sigma(s_{m_v})})$; $q_v\in P(a_{\sigma(r_1)},\ldots,a_{\sigma(r_{m_v})};a_{s_1},\ldots,a_{s_{m_v}})$.

The property is equivalent to $a_{i_{p_u(t)}}=a_{\sigma(j_t)}$, $a_{\sigma(i_{q_u(t)})}=a_{j_t}$ for $t=1,\ldots,n$ and $a_{r_{p_v(t)}}=a_{\sigma(s_t)}$, $a_{\sigma(r_{q_v(t)})}=a_{s_t}$ for $t=1,\ldots,m$. On the set $\{1,\ldots,r\}$ construct the equivalence relation $R(\sigma,f,p_u,q_u,p_v,q_v)$ generated by $i_{p_u(t)}\sim \sigma(j_t)$, $\sigma(i_{q_u(t)})\sim j_t$ for $t=1,\ldots,m_u$ and $r_{p_v(t)}\sim \sigma(s_t)$, $\sigma(r_{q_v(t)})\sim s_t$ for $t=1,\ldots,m_v$. Then indices $a_1,\ldots,a_r$ have to be equal on equivalence classes in $R(\sigma,f,p_u,q_u,p_v,q_v)$, and there is no other restriction. As such the number of tuples with the required property is $N^{\#R(\sigma,f,p_u,q_u,p_v,q_v)}$. In the end:
\[I_\gamma^N=\sum_{\substack{p_u,q_u\in Sym(m_u);\\ p_v,q_v\in Sym(m_v)}}W^N_{p_uq_u^{-1}}W^N_{p_vq_v^{-1}}N^{\#R(\sigma,f,p_u,q_u,p_v,q_v)}.\]

We fix again $p_u,q_u\in \s(m_u)$ and $p_v,q_v\in \s(m_v)$. Let $\s(f)=\{\vp\in \s(r):\vp(f^{-1}(x))=f^{-1}(x^*)\mbox{ for }x\in\{u,u^*,v,v^*\}\}$. We used the convention $x^{**}=x$. 
Then $p_u,q_u,p_v,q_v$ define an element $\vp\in Sym(f)$ as follows:
\begin{align*}
	&\vp(i_{p_u(t)})=j_t;\ \vp(j_t)=i_{q_u(t)}\mbox{ for }t=1,\ldots,m_u\\
	&\vp(r_{p_v(t)})=s_t;\ \vp(s_t)=r_{q_v(t)}\mbox{ for }t=1,\ldots,m_v.
\end{align*}
An element $\vp\in Sym(f)$ contains the same information as the tuple $(p_u,q_u,p_v,q_v)$. If we fix $\vp\in Sym(f)$ we can recover these permutations by reversing the above formulas. So, for $\vp\in Sym(f)$ define $p_u(\vp)\in Sym(n)$ by: $p_u(\vp)(t)$ is the unique $s\in\{1,\ldots,m_u\}$ such that $\vp(i_s)=j_t$. Same definition for $q_u(\vp),p_v(\vp)$ and $q_v(\vp)$.

Now $\sigma\vp\in \s(r)$. Moreover, the relation $R(\sigma,f,p_u,q_u,p_v,q_v)$ is generated by equivalences $\sigma\vp(i)\sim i$ for $i=1,\ldots,r$. Indeed $\sigma\vp(i_{p_u(t)})=\sigma(j_t)$ and $\sigma\vp(j_t)=\sigma(i_{q_u(t)})$. Similarly $\sigma\vp(r_{p_v(t)})=\sigma(s_t)$ and $\sigma\vp(s_t)=\sigma(r_{q_v(t)})$. As such $\#R(\sigma,f,p_u,q_u,p_v,q_v)$ is the number of cycles of $\sigma\vp\in Sym(r)$. It follows:
\[I_\gamma^N=\sum_{\vp\in Sym(f)}W^N_{p_u(\vp)q_u^{-1}(\vp)}W^N_{p_v(\vp)q_v^{-1}(\vp)}N^{\#cyc(\sigma\vp)}.\]

Recall that $\gamma=\lambda^{-1}\sigma\lambda$ is a permutation that contains full information on the words $w_1,\ldots,w_k$. As such, we need to compute $\lambda^{-1}\vp\lambda$. This is where the flip appears in this line of computation. Recall that $\delta_m\in\s(2m)$ is defined by: $\delta_m(t)=m+t; \delta_m(m+t)=t$ for $t\in\{1,\dots,m\}$, and
 $\delta\in\s(r)$  is the permutation $\delta_{m_u}\oplus\delta_{m_v}$.

\begin{p}
We have: $\lambda^{-1}\vp\lambda=\delta\cdot\big(p_u(\vp)^{-1}\oplus q_u(\vp)\oplus p_v(\vp)^{-1}\oplus q_v(\vp)\big)$.
\end{p}
\begin{proof}
Let $\omega(\vp)=p_u(\vp)^{-1}\oplus q_u(\vp)\oplus p_v(\vp)^{-1}\oplus q_v(\vp)$. For $t\in\{1,\dots,m_u\}$:
\begin{align*}
&\lambda^{-1}\vp\lambda(t)=\lambda^{-1}\vp(i_t)=\lambda^{-1}(j_{p_u^{-1}(t)})=m_u+p_u^{-1}(t)=\delta\omega(\vp)(t)\\
&\lambda^{-1}\vp\lambda(m_u+t)=\lambda^{-1}\vp(j_t)=\lambda^{-1}(i_{q_u(t)})=q_u(t)=\delta\omega(\vp)(m_u+t).	
\end{align*}
Similarly, for $t\in\{1,\dots,m_v\}$:
\begin{align*}
&\lambda^{-1}\vp\lambda(2m_u+t)=\lambda^{-1}\vp(r_t)=\lambda^{-1}(s_{p_v^{-1}(t)})=2m_u+m_v+p_v^{-1}(t)=\delta\omega(\vp)(2m_u+t)\\
&\lambda^{-1}\vp\lambda(2m_u+m_v+t)=\lambda^{-1}\vp(s_t)=\lambda^{-1}(r_{q_v(t)})=2m_u+q_v(t)=\delta\omega(\vp)(2m_u+m_v+t).	
\end{align*}
\end{proof}

Coming back to the formula of $I_\gamma^N$, $\# cyc(\sigma\vp)=\# cyc(\lambda^{-1}\sigma\lambda\lambda^{-1}\vp\lambda)=\# cyc(\gamma\delta\omega(\vp))$. As such:
\[I_\gamma^N=\sum_{\vp\in Sym(f)}W^N_{p_u(\vp)q_u^{-1}(\vp)}W^N_{p_v(\vp)q_v^{-1}(\vp)}N^{\#cyc(\gamma\delta\omega(\vp))}= \sum_{\substack{p_u,q_u\in Sym(m_u);\\ p_v,q_v\in Sym(m_v)}} W^N_{p_uq_u^{-1}}W^N_{p_vq_v^{-1}}N^{\#cyc(\gamma\delta\cdot p_u^{-1}\oplus q_u\oplus p_v^{-1}\oplus q_v)}\]

We record the last step of the computation as a separate Proposition, for a later use.

\begin{p}\label{unitary final step}
We have:
\[I_\gamma^N=\sum_{\substack{p_u,q_u\in Sym(m_u);\\ p_v,q_v\in Sym(m_v)}} W^N_{p_uq_u^{-1}}W^N_{p_vq_v^{-1}}N^{\#cyc(\gamma\delta\cdot p_u^{-1}\oplus q_u\oplus p_v^{-1}\oplus q_v)}=
\tau\big(\Phi_r^N(\Phi_{m_u}^N\oplus Id_{m_u+m_v}\oplus\Phi_{m_v}^N)^{-1}(\ve_{m_u}\oplus\ve_{m_v})\gamma\delta\big).\]
\end{p}
\begin{proof}
First notice that, for $W\in\cz$ and $p\in\s(r)$, we have:
\[W\cdot N^{\#cyc(p)}=\tau(N^{\#cyc(p^{-1})}p^{-1}\cdot Wp)=\tau(\Phi_r^NWp).\]
As such:
\[\sum_{\substack{p_u,q_u\in Sym(m_u);\\ p_v,q_v\in Sym(m_v)}} W^N_{p_uq_u^{-1}}W^N_{p_vq_v^{-1}}N^{\#cyc(\gamma\delta\cdot p_u^{-1}\oplus q_u\oplus p_v^{-1}\oplus q_v)}=
\tau(\Phi_r^N\sum_{\substack{p_u,q_u\in Sym(m_u);\\ p_v,q_v\in Sym(m_v)}} W^N_{p_uq_u^{-1}}W^N_{p_vq_v^{-1}}(p_u^{-1}\oplus q_u\oplus p_v^{-1}\oplus q_v)\gamma\delta).\]

Our goal is now to compute $\sum_{p,q\in Sym(m)}W^N_{pq^{-1}}p^{-1}\oplus q$ as an element of $L(\s(2m))$. Recall that $\ve_m=\sum_{p\in\s(m)}p\oplus p^{-1}$, and $P_{\ve_m}$ is the Casimir element of $L(\s(m))\subset B(L^2(\s(m)))$.

\begin{align*}
\sum_{p,q\in Sym(m)}W^N_{pq^{-1}}p^{-1}\oplus q=& \sum_{\lambda,q\in Sym(m)}W_\lambda^N(\lambda q)^{-1}\oplus q=(\sum_qq^{-1}\oplus q)(\sum_\lambda W_\lambda^N\lambda^{-1}\oplus Id_{m})\\
=&\ve_m((\Phi_m^N)^{-1}\oplus Id_m)=(\Phi_m^N\oplus Id_m)^{-1}\ve_m=(Id_m\oplus\Phi_m^N)^{-1}\ve_m.
\end{align*}
As a consequence: 
\[\sum_{\substack{p_u,q_u\in Sym(m_u);\\ p_v,q_v\in Sym(m_v)}} W^N_{p_uq_u^{-1}}W^N_{p_vq_v^{-1}}(p_u^{-1}\oplus q_u\oplus p_v^{-1}\oplus q_v)=(\Phi_{m_u}^N\oplus Id_{m_u+m_v}\oplus\Phi_{m_v}^N)^{-1}(\ve_{m_u}\oplus\ve_{m_v}).\]
This implies the conclusion.
\end{proof}

In the end we reach the formula:
\[\int_{\unit(N)^2}\Pi_{c=1}^kTr_{B(H)}(w_c(u,v))d\mu(u)d\mu(v)=\tau\big(\Phi_r^N(\Phi_{m_u}^N\oplus Id_{m_u+m_v}\oplus\Phi_{m_v}^N)^{-1}(\ve_{m_u}\oplus\ve_{m_v})\gamma\delta\big).\]

\subsection{The Casimir element in the unitary case}

As before $\rho_m:\unit(N)\to B(H^{\otimes m})$ is the representation $\rho_m(u)=u^{\otimes m}$. Let $A_m=\big(\rho_m(\unit(N))\big)^\prime$, and denote by $\Omega_m$ its Casimir element. Also $\Phi_m^N\in L(\s(m))$ is the element $\Phi_m^N=\sum_{\xi\in\s(m)}N^{\# cyc(\xi)}\xi$, and  $(\Phi_m^N)^{-1}=\sum_{\xi\in\s(m)}W_\xi^N\xi$. For $\xi\in \s(m)$, $P_\xi\in B(H^{\otimes m})$ is defined by:
\[ P_\xi(x_1\otimes \cdots \otimes x_m)=x_{\xi^{-1}(1)}\otimes \cdots \otimes x_{\xi^{-1}(m)}.\]
The family $\{P_\xi:\xi\in\s(m)\}$ spans $A_m$ by Schur--Weyl duality \cite{FH91}. It is a basis if and only if $N\geqslant m$; see also \cite[Proposition 1.10]{Pro21}.

\begin{p}\label{p:initial trace}
We have: $Tr_{B(H^{\otimes m})}(P_\xi)=N^{\# cyc(\xi)}$.
\end{p}

\begin{proof}
	Let $e_1,\ldots,e_N$ be a basis of $H$. Let $[N]^{[m]}=\{f:\{1,\ldots,m\}\to \{1,\ldots,N\}\}$. Then:
	\begin{align*}
		Tr_{B(H^{\otimes m})}(P_\xi)=&\sum_{f\in [N]^{[m]}}\langle P_\xi(e_{f(1)}\otimes\ldots\otimes e_{f(m)}),e_{f(1)}\otimes\ldots\otimes e_{f(m)}\rangle\\
		=&\sum_{f\in [N]^{[m]}}\langle e_{f(\xi^{-1}(1))}\otimes\ldots\otimes e_{f(\xi^{-1}(m))},e_{f(1)}\otimes\ldots\otimes e_{f(m)}\rangle=\sum_{f\in [N]^{[m]}}\Pi_{i=1}^m\delta_{f(\xi^{-1}(i))}^{f(i)}.
	\end{align*}
	The product $\Pi_{i=1}^m\delta_{f(\xi^{-1}(i))}^{f(i)}$ is non zero iff $f(\xi^{-1}(i))=f(i)$ for all $i=1,\ldots,m$. This means that $f$ has to take the same value on every cycle of $\xi$. This happens exactly for $N^{\# cyc(\xi)}$ elements of $[N]^{[m]}$.
\end{proof}

\begin{p}
The Casimir element is given by the formula $\Omega_m=\sum_{\xi,\eta\in Sym(m)}W_{\eta^{-1}\xi}^NP_\xi\otimes P_{\eta^{-1}}$.
\end{p}
\begin{proof}
As $\sum_\xi W_\xi^NN^{\# cyc(\xi)}=1$ and $\sum_\xi W_\xi^NN^{\# cyc(\xi\eta)}=0$ for any $\eta\neq Id$, we get $w_{\xi,\eta}=W_{\eta^{-1}\xi}^N$.
	
By construction $u^\xi=\sum_{\eta\in\s(m)}w_{\xi,\eta}P_\eta=\sum_{\eta\in\s(m)}W_{\eta^{-1}\xi}^NP_\eta$. As such: 
\[\Omega_m=\sum_{\xi,\eta}P_\xi\otimes (W_{\eta^{-1}\xi}^NP_\eta)^*=\sum_{\xi,\eta}W_{\eta^{-1}\xi}^NP_\xi\otimes P_{\eta^{-1}}.\]
\end{proof}

\begin{te}We have:
\[\int_{\unit(N)^2}\Pi_{c=1}^kTr_{B(H)}(w_c(u,v))d\mu(u)d\mu(v)=\tau\big(\Phi_r^N(\Phi_{m_u}^N\oplus Id_{m_u+m_v}\oplus\Phi_{m_v}^N)^{-1}(\ve_{m_u}\oplus\ve_{m_v})\gamma\delta\big).\] 
\end{te}
\begin{proof}
Let $I_\gamma^N=\int_{\unit(N)^2}\Pi_{c=1}^kTr_{B(H)}(w_c(u,v))d\mu(u)d\mu(v)$. By Theorem \ref*{te:main result}:
 \[I_\gamma^N=Tr_{B(H^{\otimes r})}\big( P_\gamma P_\delta(\Omega_{m_u}\otimes\Omega_{m_v})\big)=Tr_{B(H^{\otimes r})}\big((\Omega_{m_u}\otimes\Omega_{m_v})P_\gamma P_\delta\big).\] 
By the previous Propostion:
\[I_\gamma^N=\sum_{\substack{p_u,q_u\in Sym(m_u);\\ p_v,q_v\in Sym(m_v)}} W^N_{p_uq_u^{-1}}W^N_{p_vq_v^{-1}}Tr_{H^{\otimes r}}( P_{p_u}\otimes P_{q_u^{-1}}\otimes P_{p_v}\otimes P_{q_v^{-1}}\cdot P_\gamma P_\delta)=\sum_{\substack{p_u,q_u\in Sym(m_u);\\ p_v,q_v\in Sym(m_v)}} W^N_{p_uq_u^{-1}}W^N_{p_vq_v^{-1}}N^{\#cyc( p_u\oplus q_u^{-1}\oplus p_v\oplus q_v^{-1}\cdot\gamma\delta)}.\]
The conclusion follows by Propostion \ref{unitary final step}.
\end{proof}

\section{Minimal projections}

We discuss the unitary case in the general setting, i.e. there is no relation between $N$ and $r$. As such $\Phi_{m_u}^N$ might not be invertible and the set $\{P_\xi:\xi\in\s(m)\}$ generates $A_m$, but it might not be a basis. Computing the Casimir element in this case heavily relies on the representation theory of the symmetric group.

\subsection{Representation theory of $\s(m)$}

Let $Y_m$ be the set of partitions of $m$. For $\alpha\in Y_m$:
\[\alpha=(\alpha_1,\alpha_2,\ldots);\ \alpha_1\geqslant\alpha_2\geqslant\ldots;\ \sum_i\alpha_i=m.\]

The \emph{Young diagram} of $\alpha$, denoted by $D_\alpha$ consists of boxes $(i,j)$ with row index $i$ and column index $j$, satisfying $j\leqslant\alpha_i$. For $(i,j)\in D_\alpha$ the \emph{content} of $(i,j)$ is defined as $c(i,j)=j-i$. For $t\in\zz$ define:
\[s_t(\alpha)=\#\{(i,j)\in D_\alpha:c(i,j)=t\}.\]

\begin{de}
	The \emph{content polynomial} of $\alpha$ is:
	\[f_\alpha(N)=\Pi_{(i,j)\in D_\alpha}(N+j-i)=\Pi_{t\in\zz}(N+t)^{s_t(\alpha)}.\]
\end{de}

\subsection{Casimir element in the non-stable case}

\begin{lemma}
Construct $\alpha_m^N:L(\s(m))\to A_m$, $\alpha_m^N(\xi)=P_\xi$. Then $\alpha$ is a surjective homomorphism.
\end{lemma}

\begin{p}\label{p:up trace}
For any $x\in L(\s(m))$: $Tr_{B(H^{\otimes m})}(\alpha_m^N(x))=\tau(\Phi_m^Nx)$.
\end{p}
\begin{proof}
By Propostion \ref{p:initial trace}, $Tr_{B(H^{\otimes m})}(\alpha_m^N(p))=\tau(\Phi_m^Np)$, for any $p\in\s(m)$. The result follows by liniarity.
\end{proof}

Partitions $\lambda\in Y_m$ are in one-to-one corespondence with the set of ireducible representations of $\s(m)$, that are in one-to-one correspondence with the set of central projections in $L(\s(m))$.

\begin{nt}
For $\lambda\in Y_m$, let $p_\lambda$ be the central projection of $L(\s(m))$ corresponding to $\lambda$.
\end{nt}

As the Formanek element $\Phi_m^N\in L(\s(m))$ is central, $\Phi_m^N$ is a linear combination of $p_\lambda$, $\lambda\in Y_m$. The next Propostion follows from Schur--Weyl duality and the hook-content formula;
see \cite[Proposition 1.12]{Pro21} and \cite[Corollary 7.21.4]{Sta99}.

\begin{p} 
We have $\Phi_m^N=\sum_{\lambda\in Y_m}f_\lambda(N)p_\lambda$.
\end{p}

\begin{de}
Denote by $W_m^N$ the partial inverse of $\Phi_m^N$: if $Y_m^N=\{\lambda\in Y_m:f_\lambda(N)\neq 0\}$ then:
\[W_m^N=\sum_{\lambda\in Y_m^N}\frac1{f_\lambda(N)}p_\lambda.\]
\end{de}

\begin{p}
We have $\Phi_m^NW_m^N\Phi_m^N=\Phi_m^N$. If $N\geqslant m$ then $W_m^N=(\Phi_m^N)^{-1}$, as before.
\end{p}
\begin{proof}
The first part is the definition of a partial inverse, and it is easy to check. The second part follows by the invertability of the Formanek element in this case.
\end{proof}

Now we are ready to compute the Casimir element of $A_m$. The next Propostion is a particular case of Propostion \ref{p:Casimir two traces}. Notice that $(\alpha_m^N)^{\otimes 2}(x\otimes y)=\alpha_{2m}^N(x\otimes y)$.

\begin{p}
We have: $\Omega_m=(\alpha_{2m}^N)\big((W_m^N\otimes Id_m)\ve_m\big)$.
\end{p}
\begin{proof}
Let $x,y\in L(\s(m))$. We compute $V=Tr_{B(H^{\otimes 2m})}\big((\alpha_m^N)^{\otimes 2}\big((W_m^N\otimes Id_m)\ve_m\big)\cdot\alpha_m^N(x)\otimes\alpha_m^N(y)\big)$. Notice that $\alpha_m^N(x)\otimes\alpha_m^N(y)=\alpha_{2m}^N(x\otimes y)$. Then:
\begin{align*}
V=&\sum_{p\in\s(m)}Tr_{B(H^{\otimes 2m})}\big(\alpha_{2m}^N\big((W_m^N\otimes Id_m)p\otimes p^{-1}\big)\cdot\alpha_{2m}^N(x\otimes y)\big)\\
=&\sum_{p\in\s(m)}Tr_{B(H^{\otimes 2m})}\big(\alpha_{2m}^N\big(W_m^Npx\otimes p^{-1}y\big)\big)\\ 
=&\sum_{p\in\s(m)}Tr_{B(H^{\otimes m})}\big(\alpha_m^N(W_m^Npx)\big)\cdot Tr_{B(H^{\otimes m})}(\alpha_m^N(p^{-1}y))\\
=&\sum_{p\in\s(m)}\tau(\Phi_m^NW_m^Npx)\cdot\tau(\Phi_m^Np^{-1}y)=\sum_{p\in\s(m)}\tau^{\otimes 2}(p\Phi_m^NW_m^Nx\otimes p^{-1}\Phi_m^Ny)\\
=&\tau^{\otimes 2}\big(\ve_m(\Phi_m^NW_m^Nx\otimes\Phi_m^Ny)\big)= \tau(\Phi_m^NW_m^N\Phi_m^Nxy)=\tau(\Phi_m^Nxy)=Tr_{B(H^{\otimes m})}(\alpha_m^N(xy)).
\end{align*} 
The conclusion is that $(\alpha_m^N)^{\otimes 2}\big((W_m^N\otimes Id_m)\ve_m\big)$ satisfies the defining property of the Casimir element of $A_m$, so it must be equal to $\Omega_m$.
\end{proof}

\begin{p} \label{p:centralxCasimir}
We have:
$(W_m^N\otimes Id_m)\ve_m=(Id_m\otimes W_m^N)\ve_m=\sum_{\lambda\in Y_m^N}\frac1{f_\lambda(N)}(p_\lambda\otimes p_\lambda)\ve_m$.
\end{p}
\begin{proof}
Let $x,y\in L(\s(m))$. Then:
\begin{align*}
\tau\big((x\otimes y)(W_m^N\otimes Id_m)\ve_m\big)=&\tau(xW_m^Ny)=\tau(xyW_m^N);\\
\tau\big((x\otimes y)(Id_m\otimes W_m^N)\ve_m\big)=&\tau(xyW_m^N);\\
\tau\big((x\otimes y)\sum_{\lambda\in Y_m^N}\frac1{f_\lambda(N)}(p_\lambda\otimes p_\lambda)\ve_m\big)=&\sum_{\lambda\in Y_m^N}\frac1{f_\lambda(N)}\tau(xp_\lambda yp_\lambda)\\
=&\sum_{\lambda\in Y_m^N}\frac1{f_\lambda(N)}\tau(xyp_\lambda)=\tau(xyW_m^N).
\end{align*}
As such $\tau\big((x\otimes y)(W_m^N\otimes Id_m)\ve_m\big)=\tau\big((x\otimes y)(Id_m\otimes W_m^N)\ve_m\big)= \tau\big((x\otimes y)\sum_{\lambda\in Y_m^N}1/{f_\lambda(N)}\cdot(p_\lambda\otimes p_\lambda)\ve_m\big)$ for each $x,y\in L(\s(m))$. The conclusion hence follows.
\end{proof}

\subsection{Formulas for $I_\gamma^N$} We now use this computations of the Casimir elements to get new formulas for the expected value of $\Pi_{c=1}^kTr_{B(H)}(w_c(u,v))$.

\begin{te} We have:
	\[\int_{\unit(N)^2}\Pi_{c=1}^kTr_{B(H)}(w_c(u,v))d\mu(u)d\mu(v)=\tau\big(\Phi_r^N(W_{m_u}^N\otimes Id_{m_u+m_v}\otimes W_{m_v}^N)(\ve_{m_u}\otimes\ve_{m_v})\delta\gamma\big).\] 
\end{te}
\begin{proof}
By Theorem \ref{te:main result}, $I_\gamma^N=Tr_{B(H^{\otimes r})}\big((\Omega_{m_u}\otimes\Omega_{m_v}) P_\delta P_\gamma\big)$. Using the previous Propostions, $\Omega_{m_u}\otimes\Omega_{m_v}=\alpha_r^N\big((W_{m_u}^N\otimes Id_{m_u+m_v}\otimes W_{m_v}^N)\ve_{m_u}\otimes\ve_{m_v}\big)$, so $(\Omega_{m_u}\otimes\Omega_{m_v})P_\gamma P_\delta=\alpha_r^N\big((W_{m_u}^N\otimes Id_{m_u+m_v}\otimes W_{m_v}^N)(\ve_{m_u}\otimes\ve_{m_v})\delta\gamma \big)$. The conclusion follows now from Propostion \ref{p:up trace}.
\end{proof}

\begin{p} We have:
\[\int_{\unit(N)^2}\Pi_{c=1}^kTr_{B(H)}(w_c(u,v))d\mu(u)d\mu(v)=\sum_{\mu\in Y_r}\sum_{\substack{\lambda_1\in Y_{m_u}^N\\ \lambda_2\in Y_{m_v}^N}} \frac{f_\mu(N)}{f_{\lambda_1}(N)f_{\lambda_2}(N)}\tau\big(p_\mu(p_{\lambda_1}\otimes p_{\lambda_1}\otimes p_{\lambda_2}\otimes p_{\lambda_2})(\ve_{m_u}\otimes\ve_{m_v})\delta\gamma\big).\]
\end{p}
\begin{proof}
Recall that $\Phi_r^N=\sum_{\mu\in Y_r}f_\mu(N)p_\mu$. Moreover, by Propostion \ref{p:centralxCasimir}:
\[(W_{m_u}^N\otimes Id_{m_u+m_v}\otimes W_{m_v}^N)(\ve_{m_u}\otimes\ve_{m_v})=\sum_{\lambda_1\in Y_{m_u}^N;\lambda_2\in Y_{m_v}^N}\frac1{f_{\lambda_1}(N)f_{\lambda_2}(N)}(p_{\lambda_1}\otimes  p_{\lambda_1}\otimes p_{\lambda_2}\otimes p_{\lambda_2})(\ve_{m_u}\otimes\ve_{m_v}).\]
The result now follows from the previous Theorem.
\end{proof}

\subsection{Set of possible matchings}

We end this section with a description of the product $\ve_m\delta_m$.

\begin{de}
Denote by $T_m=\{\xi\in\s(2m):\xi^2=Id;\xi(\{1,\ldots,m\})=\{m+1,\ldots,2m\}\}$.
\end{de}

Notice that $\delta_m\in T_m$, and actually $T_m$ is the set of all possible "$\delta$''.

\begin{p}
We have $\ve_m\delta_m=\sum_{\xi\in T_m}\xi$.
\end{p}
\begin{proof}
Let $\xi\in T_m$. Then there exists a unique $\sigma\in\s(m)$ such that $\xi(i)=\sigma(i)+m$ and $\xi(i+m)=\sigma^{-1}(i)$. As such $\xi\delta_m(i)=\sigma^{-1}(i)$ and $\xi\delta_m(i+m)=\sigma(i)$. It follows that $\xi\delta_m=\sigma^{-1}\oplus \sigma$. 

Then $\sum_{\xi\in T_m}\xi\delta_m=\sum_{\sigma\in\s(m)}\sigma^{-1}\otimes \sigma=\ve_m$ and hence the conclusion.
\end{proof}

\section{Vector space of rational functions}

Let $\mu\in Y_r$, $\lambda_1\in Y_{m_u}$, $\lambda_2\in Y_{m_v}$. Define the rational function:
\[W_{\mu;\lambda_1,\lambda_2}(N)=\frac{f_\mu(N)}{f_{\lambda_1}(N)f_{\lambda_2}(N)}.\]

We want to estimate the dimension of the linear space generated by thse functions.

\subsection{The least common multiple polynomial}

Let $\alpha\in Y_m$. Recall that $f_\alpha(N)=\Pi_{(i,j)\in D_\alpha}(N+j-i)=\Pi_{t\in\zz}(N+t)^{s_t(\alpha)}$, so $f$ is a polynomial of degree $m$. As each Young diagram contains the box $(1,1)$ of content $0$, $f$ has no free coeficient. 

Let $D_m$ be the least common multiple of polynomials $f_\lambda$, $\lambda\in Y_m$. Let $s_t=\max\{s_t(\alpha):\alpha\in Y_m\}$. Then:
\[D_m(N)=lcm_{\lambda\in Y_m}f_\lambda(N)=\Pi_{t\in\zz}(N+t)^{s_t}.\]
As such the degree of $D_m$ is $\sum_{t=-m}^ms_t$. We now estimate the degree of $D_m$.

\begin{lemma}
For any $t\in\zz$, $s_t(s_t+|t|)\leqslant m$.
\end{lemma}
\begin{proof}
Assume $t\geqslant 0$ and let $\alpha$ such that $s_t=s_t(\alpha)$. The first cell of content $t$ is $(1,t+1)$, then $(2,t+2)$ and so on. As $s_t(\alpha)=s_t$ it follows that $(s_t,t+s_t)\in D_\alpha$. This implies $s_t(s_t+t)\leqslant m$. The proof is the same in the case $t<0$.
\end{proof}

\begin{lemma}
We have $\sum_{t>0}s_t\leqslant\frac12mln(m)+\sqrt{m}$. 
\end{lemma}
\begin{proof}
Let $A=\{(s,t):s\geqslant 1;t\geqslant 1;s^2+st\leqslant m\}$. Then $\sum_{t>0}s_t=Card(A)$.

Let $A_s=\{t:t\geqslant 1;s^2+st\leqslant m\}$. Then $Card(A)=\sum_{s>0} Card(A_s)$ and $Card(A_s)=max\{0,\lfloor\frac{m-s^2}{s}\rfloor\}$. As such:
\begin{align*}
Card(A)=&\sum_{s=1}^{\lfloor\sqrt{m}\rfloor}Card(A_s)=\sum_{s=1}^{\lfloor\sqrt{m}\rfloor}\lfloor\frac ms\rfloor-\sum_{s=1}^{\lfloor\sqrt{m}\rfloor}s\leqslant\big(m\sum_{s=1}^{\lfloor\sqrt{m}\rfloor}\frac1s\big)-(m-\sqrt{m})\\
=&\leqslant m(\frac12ln(m)+1)-(m-\sqrt m)=\frac12mln(m)+\sqrt m.
\end{align*}
\end{proof}

\begin{p}\label{p:degree of lcm}
The degree of $D_m$ is less than $mln(m)+3\sqrt m$.
\end{p}
\begin{proof}
The degree of $D_m$ is $\sum_ts_t=s_0+\sum_{t<0}s_t+\sum_{t>0}s_t$. Using the two previous lemmas:
\[\sum_ts_t\leqslant \sqrt m+\frac12mln(m)+\sqrt{m}+\frac12mln(m)+\sqrt{m}=mln(m)+3\sqrt m.\]
\end{proof}

\subsection{The space of $W$ functions}

\begin{te}
The space generated by $\{W_{\mu;\lambda_1,\lambda_2}:\mu\in Y_r, \lambda_1\in Y_{m_u}, \lambda_2\in Y_{m_v}\}$ has dimension less than $r(ln(r)+2)$.
\end{te}
\begin{proof}
Recall that $W_{\mu;\lambda_1,\lambda_2}=f_\mu/(f_{\lambda_1}f_{\lambda_2})$. Then $W_{\mu;\lambda_1,\lambda_2}D_{m_u}D_{m_v}$ is a polynomial of degree: \[deg(D_{m_u})+deg(D_{m_v})+deg(f_\mu)-deg(f_{\lambda_1})-deg(f_{\lambda_2}).\]
As $deg(f_\lambda)=m$ for $\lambda\in Y_m$ and by Proposition \ref{p:degree of lcm} we have:
\[deg(W_{\mu;\lambda_1,\lambda_2}D_{m_u}D_{m_v})\leqslant m_uln(m_u)+3\sqrt{m_u}+m_vln(m_v)+3\sqrt{m_v}+m_u+m_v<r(ln(r)+2).\]
\end{proof}

\section{Weingarden numbers via paths}

In this section we present a method to compute Weingarden numbers and product of Weingarden numbers using \emph{paths} between bijective functions. The underlying permutation-factorization expansion is due to Collins \cite[Theorem 2.2]{Col03}; a simmilar approach is used in   Magee--Puder \cite[Theorem 2.9]{MP19}.

\begin{nt}
Let $A_1,A_2$ be two sets of equal cardinality. We denote by $Bij(A_1,A_2)$ the set of bijective functions $f:A_1\to A_2$.
\end{nt}

Recall that for a permutation $p\in\s(m)$, $\#cyc(p)$ is the number of cycles in $p$ and its transposition lenght is $|p|=m-\#cyc(p)$.

\begin{ob}
If $p,q\in Bij(A_1,A_2)$ then $p^{-1}q\in Bij(A_1,A_1)\simeq\s(m)$, where $m=Card A_1$. As such we can define $\#cyc(p^{-1}q)$ and $|p^{-1}q|$. We also use the notation $p\oplus q^{-1}$, which is an element in $\s(A_1\sqcup A_2)$.
\end{ob}

Fix $N,m\in\nz^*$ and $A_1,A_2$ two sets of cardinality $m$. Let $K_m\in M_{m!}(\rz)$, $K_m(p,q)=N^{\#cyc(p^{-1}q)}$, for $p,q\in Bij(A_1,A_2)$. Then, by the definition of the Weingarden numbers, $K_m^{-1}(p,q)=W_{p^{-1}q}^N$. 

Notice that $K_m=N^mG_m$, where $G_m(p,q)=N^{-|p^{-1}q|}$. As such $G_m(p,p)=1$. Let $G_m=Id_{m!}+S_m$, so that $S_m(p,p)=0$. Also $K_m^{-1}=N^{-m}G_m^{-1}$ and:
\[G_m^{-1}=(Id_{m!}+S_m)^{-1}=\sum_{i\geqslant 0}(-1)^iS_m^i=Id_{m!}-S_m+S_m^2-S_m^3+\ldots\]
The series $\sum_{i\geqslant 0}(-1)^iS_m^i$ is convergent only for large enough $N$. Note that, the stable case inequality, $N\geqslant r$, might not be enough. From now on we assume convergence of these series. This equation enables us a new way to compute the Weingarden numbers.

\subsection{Paths}

The following defintion is motivated by our need to compute $S_m^i(p,q)$, for $p,q$ two bijective functions. 

We write the inverse-Gram expansion as paths between bijections. The underlyingpermutation-factorization expansion is due to Collins [Col03, Theorem 2.2]; itsmatching-chain form occurs in Magee–Puder [MP19, Theorem 2.9 and (2.11)].

\begin{de}\label{de:path}
Let $p,q\in Bij(A_1,A_2)$. A \emph{path} $T$ between $p$ and $q$ is a finite sequence of elements in $Bij(A_1,A_2)$, $\xi_0,\xi_1,\ldots,\xi_t$ such that:
	\begin{enumerate}
		\item $\xi_0=p$ and $\xi_t=q$;
		\item $\xi_{i-1}\neq\xi_i$ for $i=1,\ldots,t$.
	\end{enumerate}

We denote by $Ph(p,q)$ the set of paths between $p$ and $q$. For a path $T$, we denote by $\ve(T)$ its lenght, $\ve(T)=t$. Define $L(T)=\sum_{i=1}^t|\xi_{i-1}^{-1}\xi_i|$. When $p = q$, by convention, there is a path $T_{p,p}=(p,p)$ with no transitions, with $\ve(T) = 0$ and $L(T) = 0$.
\end{de}

\begin{p}\label{p:weingarden1}
For $p,q\in Bij(A_1,A_2)$, we have $W_{p^{-1}q}^N=N^{-m}\sum_{T\in Ph(p,q)}(-1)^{\ve(T)}N^{-L(T)}$.
\end{p}
\begin{proof}
Recall that: 
\[W_{p^{-1}q}^N=K_m^{-1}(p,q)=N^{-m}G_m^{-1}(p,q)=N^{-m}\sum_{i\geqslant 0}(-1)^iS_m^i(p,q).\]
We denote by $T=(\xi_0,\ldots,\xi_i)$ a generic path. Then:
\[S_m^i(p,q)=\sum_{\substack{T\in Ph(p,q)\\ \ve(T)=i}}\Pi_{j=1}^iS_m(\xi_{j-1},\xi_j)=\sum_{\substack{T\in Ph(p,q)\\ \ve(T)=i}}\Pi_{j=1}^iN^{-|\xi_{j-1}^{-1}\xi_j|}=\sum_{\substack{T\in Ph(p,q)\\ \ve(T)=i}}N^{-L(T)}.\]
As such $W_{p^{-1}q}^N=N^{-m}\sum_{T\in Ph(p,q)}(-1)^{\ve(T)}N^{-L(T)}$.
\end{proof}

\subsection{Multi permutations}

We can compute product of Weingarden numbers in a similar way, using \emph{generalised paths}. Let $\mathcal{M}=(m_1,\ldots,m_s)$ be a multi index. Set $l(\mathcal{M})=\sum_im_i$. We denote by $\s(\mathcal{M})$ the group $\s(m_1)\times\ldots\times\s(m_s)\subset\s(l(\mathcal{M}))$. Let $\mathcal{M}!=m_1!\cdot\ldots\cdot m_s!$ such that $Card(\s(\mathcal{M}))=\mathcal{M}!$.

Let $\acal_i=(A_1^i,\ldots,A_s^i)$, $i=1,2$, be multisets of cardinality $\mm$, i.e. $Card(A_j^i)=m_j$. Then $Bij(\acal_1,\acal_2)$ is the set of tuples $p=(p_1,\ldots,p_s)$, where $p_i\in Bij(A_i^1,A_i^2)$. As such, if $p,q\in Bij(\acal_1,\acal_2)$, then $p^{-1}q\in Bij(\acal_1,\acal_1)\simeq\s(\mm)$.

For $p=(p_1,\ldots,p_s)\in\s(\mm)$, we have $\#cyc(p)=\sum_i\#cyc(p_i)$ and $|p|=\sum_i|p_i|$. As before, let $K_\mm,G_\mm\in M_{\mm!}(\rz)$ where $K_\mm(p,q)=N^{\#cyc(p^{-1}q)}$ and $G_\mm(p,q)=N^{-|p^{-1}q|}$. As such $K_\mm=N^{l(\mm)}G_\mm$.

Now, by definition, we can see that $K_\mm(p,q)=\Pi_{i=1}^sK_{m_i}(p_i,q_i)$. As such $K_\mm=\otimes_{i=1}^sK_{m_i}$. Then $K_\mm^{-1}=\otimes_{i=1}^sK_{m_i}^{-1}$. In the end:
\[K_\mm^{-1}(p,q)=\Pi_{i=1}^sK_{m_i}^{-1}(p_i,q_i)=W_{p_1^{-1}q_1}^NW_{p_2^{-1}q_2}^N\ldots W_{p_s^{-1}q_s}^N.\]

Let $G_\mm=Id_{\mm!}+S_\mm$, so that $S_\mm(p,p)=0$. Then $K_\mm^{-1}=N^{-l(\mm)}G_\mm^{-1}$ and:
\[G_\mm^{-1}=(Id_{\mm!}+S_\mm)^{-1}=\sum_{i\geqslant 0}(-1)^iS_\mm^i=Id_{\mm!}-S_\mm+S_\mm^2-S_\mm^3+\ldots\]

\subsection{Generalised paths}

The definition for \emph{generalised paths} is identical to that of \emph{paths}. We denote the set of generalised paths by $GPh(p,q)$, where $p,q$ are now elements in $Bij(\acal_1,\acal_2)$.

\begin{de}\label{de:generalisedpath}
	Let $p,q\in Bij(\acal_1,\acal_2)$. A \emph{generalised path} $T$ between $p$ and $q$ is a finite sequence of elements in $Bij(\acal_1,\acal_2)$, $\xi_0,\xi_1,\ldots,\xi_t$ such that:
	\begin{enumerate}
		\item $\xi_0=p$ and $\xi_t=q$;
		\item $\xi_{i-1}\neq\xi_i$ for $i=1,\ldots,t$. 
	\end{enumerate}
	The inequality $\xi_{i-1}\neq\xi_i$ indicates that at least one bijection $\xi_{i-1}^j$ within $\xi_{i-1}$ differs from $\xi_i^j$.
	We denote by $GPh(p,q)$ the set of generalised paths between $p$ and $q$. For a path $T$, we denote by $\ve(T)$ its lenght, $\ve(T)=t$. Also $L(T)=\sum_{i=1}^t|\xi_{i-1}^{-1}\xi_i|$.	
\end{de}

\begin{nt}
For $p,q\in Bij(\acal_1,\acal_2)$, we denote by $W_{p^{-1}q}^N$ the product $W_{p_1^{-1}q_1}^NW_{p_2^{-1}q_2}^N\ldots W_{p_s^{-1}q_s}^N$.
\end{nt}

The following Propostion is the generalisation of Proposition \ref{p:weingarden1}. The proof is virtually the same.

\begin{p}\label{p:productofweingarden}
	For $p,q\in Bij(\acal_1,\acal_2)$, we have $W_{p^{-1}q}^N=N^{-l(\mm)}\sum_{T\in GPh(p,q)}(-1)^{\ve(T)}N^{-L(T)}$.
\end{p}
\begin{proof}
	Recall that: 
	\[W_{p^{-1}q}^N=K_\mm^{-1}(p,q)=N^{-l(\mm)}G_\mm^{-1}(p,q)=N^{-l(\mm)}\sum_{i\geqslant 0}(-1)^iS_\mm^i(p,q).\]
	We denote by $T=(\xi_0,\ldots,\xi_i)$ a generic generalised path. Then:
	\[S_\mm^i(p,q)=\sum_{\substack{T\in GPh(p,q)\\ \ve(T)=i}}\Pi_{j=1}^iS_\mm(\xi_{j-1},\xi_j)=\sum_{\substack{T\in GPh(p,q)\\ \ve(T)=i}}\Pi_{j=1}^iN^{-|\xi_{j-1}^{-1}\xi_j|}=\sum_{\substack{T\in GPh(p,q)\\ \ve(T)=i}}N^{-L(T)}.\]
	As such $W_{p^{-1}q}^N=N^{-l(\mm)}\sum_{T\in GPh(p,q)}(-1)^{\ve(T)}N^{-L(T)}$.
\end{proof}

\subsection{The path of a generalised path}

A \emph{generalised path} is just a \emph{path} split into pieces. The path we can associate to a generalised one will play a role latter.

Let $T\in GPh(p,q)$, $p,q\in Bij(\acal_1,\acal_2)$, $\acal_x=(A_1^x,\ldots,A_s^x)$, $x=1,2$. Set $T=(\xi_0,\xi_1,\ldots,\xi_{\ve(T)})$. By definition $\xi_i=(\xi_i^1,\ldots,\xi_i^s)$, with $\xi_i^j\in Bij(A_1^j,A_2^j)$.

Let $A_x=\sqcup_{j=1}^sA_s^j$. Construct $c(\xi_i)\in Bij(A_1,A_2)$, $c(\xi_i)=\xi_i^1\oplus\ldots\oplus\xi_i^s$. Then $C(T)=(c(\xi_0),c(\xi_1),\ldots,c(\xi_{\ve(T)}))$ is a path between $c(p)$ and $c(q)$.

\begin{de}
For $T\in GPh(p,q)$, we call $C(T)\in Ph(c(p),c(q))$ \emph{the path associated} to $T$.
\end{de}

\begin{lemma}
For $T\in GPh(p,q)$ we have $\ve(C(T))=\ve(T)$ and $L(C(T))=L(T)$.
\end{lemma}
\begin{proof}
The equality $\ve(C(T))=\ve(T)$ comes from the construction of $C(T)$. The second equality is implied by $|p_1\oplus\ldots\oplus p_s|=|p_1|+\ldots+|p_s|$, where $p_i\in\s(m_i)$.
\end{proof}

\section{Connected cumulants of traces}

We now extract the connected information in the trace moments through the classical joint cumulants of $\Tr(w_1),\ldots,\Tr(w_k)$. Their inversion is over all set partitions \cite{Rot64}. This use of trace cumulants is established in the work of Mingo--\'Sniady--Speicher and Collins--Mingo--\'Sniady--Speicher \cite{MSS07,CMSS07}. The independent extensions of the Diaconis--Shahshahani theorem by R\u{a}dulescu and by Mingo--\'Sniady--Speicher describe the limiting fluctuation background \cite{DS94,Rad06,MSS07}.

We have a balance set of words $w_1,\ldots,w_k$, with $m_u$ occurences of letter $u$ and $m_v$ occurences of letter $v$. Denote by $r$ the total lenght of words: $r=2m_u+2m_v$. Let $A^u=\{1,\ldots,m_u\}$, $A^{u^*}=\{m_u+1,\ldots,2m_u\}$, $A^v=\{2m_u+1,\ldots,2m_u+m_v\}$ and $A^{v^*}=\{2m_u+m_v+1,\ldots,r\}$. For $p_u,q_u\in Bij(A^u,A^{u^*})$, $p_u\oplus q_u^{-1}$ is a permutation in $\s(2m_u)$: $(p_u\oplus q_u^{-1})(i)=p_u(i)$ for $i\in A^u$ and $(p_u\oplus q_u^{-1})(i)=q_u^{-1}(i)$ for $i\in A^{u^*}$. In the same way $p_u\oplus q_u^{-1}\oplus p_v\oplus q_v^{-1}\in\s(r)$  for $p_u,q_u\in Bij(A^u,A^{u^*})$ and $p_v,q_v\in Bij(A^v,A^{v^*})$.

\begin{p}\label{p:value with bijections}
	$I_\gamma^N=\sum_{\substack{p_u,q_u\in Bij(A^u,A^{u^*});\\ p_v,q_v\in Bij(A^v,A^{v^*})}} W^N_{p_uq_u^{-1}}W^N_{p_vq_v^{-1}}N^{\#cyc(\gamma\cdot p_u\oplus q_u^{-1}\oplus p_v\oplus q_v^{-1})}$.
\end{p}
\begin{proof}
	This is just a restatement of the formula: $I_\gamma^N=\sum_{\substack{p_u,q_u\in Sym(m_u);\\ p_v,q_v\in Sym(m_v)}} W^N_{p_uq_u^{-1}}W^N_{p_vq_v^{-1}}N^{\#cyc(\gamma\delta\cdot p_u^{-1}\oplus q_u\oplus p_v^{-1}\oplus q_v)}$, obtained in Section \ref{s:unitary case}.
\end{proof}

\subsection{Subsets of words}

Recall that cycles in $\gamma$ corresponds to the words $w_1,\ldots,w_k$. So $\gamma=c_1\ldots c_k$. Let $S=\{1,\ldots,k\}$. For $i\in S$ we denote by $\tilde c_i$ the set of numbers in $\{1,\ldots,r\}$ that appear in the cycle $c_i$. Let $A_i^u=\tilde c_i\cap A^u$, i.e. $A_i^u$ represents the $u$-entries in the cycle $c_i$. In the same way define $A_i^x=\tilde c_i\cap A^x$, $x\in\{u^*,v,v^*\}$. Also let $A_i=A_i^u\cup A_i^{u^*}\cup A_i^v\cup A_i^{v^*}$, i.e. $A_i$ is the support of the cycle $c_i$.

Let $B\subset S$. Define $\gamma_B=\Pi_{i\in B}c_i$. Also $A_B^x=\cup_{i\in B}A_i^x$, for $x\in\{u,u^*,v,v^*\}$. So $\gamma_B$ is a permutation on the set $A_B=A_B^u\cup A_B^{u^*}\cup A_B^v\cup A_B^{v^*}=\cup_{i\in B}A_i$. Set:
\[I_\gamma^N(B)=\int_{\unit(N)^2}\Pi_{c\in B}Tr(w_c(u,v))dudv.\]

If $Card (A_B^u)\neq Card(A_B^{u^*})$ or $Card (A_B^v)\neq Card(A_B^{v^*})$ then $I_\gamma^N(B)=0$, as the family of words $\{w_c:c\in B\}$ is not balanced.

\begin{p}\label{p:value of subset}
	$I_\gamma^N(B)=\sum_{\substack{p_u,q_u\in Bij(A_B^u,A_B^{u^*});\\ p_v,q_v\in Bij(A_B^v,A_B^{v^*})}} W^N_{p_uq_u^{-1}}W^N_{p_vq_v^{-1}}N^{\#cyc(\gamma_B\cdot p_u\oplus q_u^{-1}\oplus p_v\oplus q_v^{-1})}$.
\end{p}
\begin{proof}
This is Propostion \ref{p:value with bijections}, applied to the familiy of words $\{w_c:c\in B\}$.
\end{proof}

\subsection{Partitions}

We denote by $Part(S)$ the set of partitions of the set $S=\{1,\ldots,k\}$, i.e. if $\pi\in Part(S)$ then $\pi=(B_1,\ldots,B_s)$ with $\sqcup_iB_i=S$. Define $A_\pi^x=(A_{B_1}^x,\ldots,A_{B_s}^x)$, for $x\in\{u,u^*,v,v^*\}$. Let $\mm_x=(|A_{B_1}^x|,\ldots,|A_{B_s}^x|)$, for $x\in\{u,v\}$.

For $\pi\in Part(S)$ define:
\[I_\gamma^N(\pi)=\Pi_{B\in\pi}I_\gamma^N(B).\]

\begin{p}\label{p:valuewithpaths}
	$I_\gamma^N(\pi)=\sum_{\substack{p_u,q_u\in Bij(A_\pi^u,A_\pi^{u^*});\\ p_v,q_v\in Bij(A_\pi^v,A_\pi^{v^*})}} W^N_{p_uq_u^{-1}}W^N_{p_vq_v^{-1}}N^{\#cyc(\gamma\cdot p_u\oplus q_u^{-1}\oplus p_v\oplus q_v^{-1})}$.
\end{p}
\begin{proof}
Using Propostion \ref{p:value of subset} we have:
\begin{align*}
I_\gamma^N(\pi)=&\Pi_{B\in\pi}I_\gamma^N(B)=\Pi_{B\in\pi}\sum_{\substack{p_u^B,q_u^B\in Bij(A_B^u,A_B^{u^*});\\ p_v^B,q_v^B\in Bij(A_B^v,A_B^{v^*})}} W^N_{p_u^B(q_u^B)^{-1}}W^N_{p_v^B(q_v^B)^{-1}}N^{\#cyc(\gamma_B\cdot p_u^B\oplus (q_u^B)^{-1}\oplus p_v^B\oplus (q_v^B)^{-1})}\\
=&\sum_{\substack{p_u,q_u\in Bij(A_\pi^u,A_\pi^{u^*});\\ p_v,q_v\in Bij(A_\pi^v,A_\pi^{v^*})}} W^N_{p_uq_u^{-1}}W^N_{p_vq_v^{-1}}N^{\sum_{B\in\pi}\#cyc(\gamma_B\cdot p_u^B\oplus (q_u^B)^{-1}\oplus p_v^B\oplus (q_v^B)^{-1})}\\
=&\sum_{\substack{p_u,q_u\in Bij(A_\pi^u,A_\pi^{u^*});\\ p_v,q_v\in Bij(A_\pi^v,A_\pi^{v^*})}} W^N_{p_uq_u^{-1}}W^N_{p_vq_v^{-1}}N^{\#cyc(\gamma\cdot p_u\oplus q_u^{-1}\oplus p_v\oplus q_v^{-1})}.
\end{align*}
\end{proof}

\subsection{Decorated path}

\begin{de}
	Fix $\pi\in Part(S)$. Let $p=(p_u,p_v)$ and $q=(q_u,q_v)$, where $p_u,q_u\in Bij(A_\pi^u,A_\pi^{u^*})$ and $p_v,q_v\in Bij(A_\pi^v,A_\pi^{v^*})$. A \emph{decorated path} $T$ between $p$ and $q$ is a pair $T=(T_u,T_v)$, where $T_u\in GPh(p_u,q_u)$ and $T_v\in GPh(p_v,q_v)$.
	
	We denote by $DPh_\pi(p,q)$ the set of decorated paths between $p$ and $q$, and let $DPh_\pi=\cup_{p,q}DPh_\pi(p,q)$.
	For a decorated path $T$, let $\ve(T)=\ve(T_u)+\ve(T_v)$. We allow $T_u$ and $T_v$ to be of different lenghts. Let $L(T)=L(T_u)+L(T_v)$.
\end{de}

\begin{de}
Let $p=(p_u,p_v)$ and $q=(q_u,q_v)$ as before. Set $V_\gamma(p,q)=\#cyc(\gamma\cdot p_u\oplus q_u^{-1}\oplus p_v\oplus q_v^{-1})$. For $T\in DPh_\pi(p,q)$, define $\chi(T)=V_\gamma(p,q)-(m_u+m_v)-L(T)$.
\end{de}

We can now rewite the formula for $I_\gamma^N(\pi)$ in terms of decorated paths.

\begin{te} \label{te:valuewithpaths}
For any $\pi\in Part(S)$:
\[I_\gamma^N(\pi)=\sum_{T\in DPh_\pi}(-1)^{\ve(T)}N^{\chi(T)}.\]
\end{te}
\begin{proof}
This is a consequence of Propostions \ref{p:productofweingarden} and \ref{p:valuewithpaths}.
\end{proof}

\subsection{The partition of a decorated path}

\begin{nt}
	We write $1_S\in Part(S)$ for the partition consisting of only one set. Also write $DPh$ for $DPh_{1_S}$.
\end{nt}

\begin{de}
	For $T\in DPh_\pi$, $T=(T_u,T_v)$, define $C(T)\in DPh$, $C(T)=(C(T_u),C(T_v))$.
\end{de}

\begin{ob}
	Notice that $\chi(T)=\chi(C(T))$ for any decorated path $T$.
\end{ob}

Two decorated paths $T_1$ and $T_2$ are \emph{equivalent} if $C(T_1)=C(T_2)$. For a decorated path $T$, we are only interested in $C(T)$ and the minimum partition $\pi$ such that $T\in DPh_\pi$. We now formalise this partition.

\begin{de}
 For $\pi_1,\pi_2\in Part(S)$, we write $\pi_1\leqslant\pi_2$ if $\pi_1$ is a subpartition of $\pi_2$, i.e. for any $B_1\in \pi_1$ there exists $B_2\in\pi_2$ such that $B_1\subset B_2$.
\end{de}

\begin{de}
For $T\in DPh$ denote by $\rho(T)\in Part(S)$ the minimum partition $\pi$ such that there exists $T_0\in DPh_\pi$ with $C(T_0)=T$. Formally:
\[\rho(T)=\min\{\pi:\exists T_0\in DPh_\pi, C(T_0)=T\}.\]
\end{de}

We now describe $\rho(T)$ using orbits of a specific subgroup of $\s(r)$ that we associate to $T$.

\begin{de}\label{de:associated group}
Let $T=(T_u,T_v)\in DPh$. Construct $G_T\leqslant\s(r)$ by:
\[G_T=\langle \gamma,p_1\oplus p_2^{-1}\oplus p_3\oplus p_4^{-1}| p_1,p_2\in T_u;p_3,p_4\in T_v\rangle.\]
\end{de}

\begin{nt}
For $G\leqslant\s(r)$, we denote by $Orb(G)$, the orbit subsets of the action of $G$ on$\{1,\ldots,r\}$.
\end{nt}

\begin{ob}
If $\gamma\in G$, then $G$ induces a partition on $S$. Indeed each $A_i$, $i=1,\ldots,k$ is contained in an orbit of $G$. Denote by $\rho(G)$ the partition on $S$ induced this way. 
\end{ob}

\begin{p}
For any $T\in DPh$, $\rho(T)=\rho(G_T)$.
\end{p}
\begin{proof}
Construct a graph $\gr$ as follows: $S$ is the set of vertices. Let $p\in T_u$, $i\in S$. Draw an edge in $\gr$ between $i$ and any $j\in S$ such that $p_u(A_i^u)$ intersects $A_j^{u^*}$. Do the same for any $p\in T_v$. It is easy to see that the connected components of $\gr$ are given by both the partitions $\rho(T)$ and $\rho(G_T)$.
\end{proof}

\subsection{Cummulants}

\begin{te}\label{te:integral with paths}
For any $\pi\in Part(S)$:
\[I_\gamma^N(\pi)=\sum_{T\in DPh,\rho(T)\leqslant\pi}(-1)^{\ve(T)}N^{\chi(T)}.\]
\end{te}
\begin{proof}
This is Theorem \ref{te:valuewithpaths}, where paths in $DPh_\pi$ are transported to paths in $DPh$ with the associated partition smaller than $\pi$.
\end{proof}

\subsection{Moebus inversion formula}

\begin{nt}
For $\pi\in Part(S)$, denote by $F(\pi)=\sum_{T\in DPh,\rho(T)=\pi}(-1)^{\ve(T)}N^{\chi(T)}$.
\end{nt}

Theorem \ref{te:integral with paths} becomes $I_\gamma^N(\pi)=\sum_{\eta\leqslant\pi}F(\eta)$. Using an inclusion-exclusion type argument, we can find $\mu(\pi,1_S)\in\rz$ such that:
\[\sum_{\pi\in Part(S)}\mu(\pi,1_S)I_N(\pi)=F(1_S)=\sum_{T\in DPh,\rho(T)=1_S}(-1)^{\ve(T)}N^{\chi(T)}.\]

Formally, one defines \emph{the Zeta function} $\zeta:Part(S)^2\to\{0,1\}$, $\zeta(\eta,\pi)=1$ if $\eta\leqslant\pi$ and $\zeta(\eta,\pi)=0$ otherwise. Then \emph{the Moebus function} $\mu:Part(S)^2\to\rz$ is its inverse such that, for any $\eta,\pi$:
\[\zeta*\mu(\eta,\pi)=\sum_{\eta\leqslant\rho\leqslant\pi}\zeta(\eta,\rho)\mu(\rho,\pi)=\delta_\eta^\pi.\]

\begin{de}
For a set of words $w_1,\ldots,w_k$ set:
\[\kappa_\gamma^N=\sum_{\pi\in Part(S)}\mu(\pi,1_S)I_\gamma^N(\pi).\]
\end{de}

The main result of this section.
\begin{te}\label{te:formula cummuland}
For a set of words $w_1,\ldots,w_k$, we have the equation:
\[\kappa_\gamma^N=\sum_{T\in DPh,\rho(T)=1_S}(-1)^{\ve(T)}N^{\chi(T)}.\]
\end{te}

Actually, $\mu(\pi,1_S)=(-1)^{|\pi|-1}(|\pi|-1)!$ \cite{Rot64}. The resulting quantity is the classical joint cumulant
of the scalar traces \cite{Col03,MSS07,CMSS07}.

\section{Euler characteristics of decorated paths}

The goal of this section is to prove that for any $T\in DPh$ we have:
\[\chi(T)\leqslant 2|\rho(T)|-k,\]
where $|\rho(T)|$ is the number of subsets of $S$ that composes the partition $\rho(T)$. Moreover the diference is always an even integer.

This is an aplication of Rimhak Ree's Theorem, \cite{Ree71}. 

\begin{te}(Ree's Theorem)
Let $p_1,\ldots,p_s\in\s(m)$ be a collection of permutations such that $p_1\cdot p_2\cdot\ldots\cdot p_s=Id_m$. Let $t$ be the number of orbits of the action generated by $p_1,\ldots,p_s$ on $\{1,\ldots,m\}$. Then:
\[|p_1|+|p_2|+\ldots+|p_s|\geqslant 2(m-t),\]
and the diference is always an even integer.
\end{te}

\subsection{Direct endpoint paths}


\begin{nt}
Let $p=(p_u,p_v)$ and $q=(q_u,q_v)$. We denote by $T_{p,q}\in DPh(p,q)$ the \emph{direct endpoints decorated path} $T_{p,q}=(T_u,T_v)$, where $T_u=(p_u,q_u)$ and $T_v=(p_v,q_v)$. 
\end{nt}

Note that $L(T_{p,q})=|p_u^{-1}q_u|+|p_v^{-1}q_v|$. Its total lenght, $\ve(T_{p,q})$, can be zero, one, or two.
Recall that $\chi(T)=V_\gamma(p,q)-(m_u+m_v)-L(T)$.

\begin{p}\label{p: chi for 2-paths}
For $T=T_{p,q}$ we have $\chi(T)\leqslant 2|\rho(T)|-k$, and the difference is an even integer.
\end{p}
\begin{proof}
We apply Ree's theorem for the permutations $\gamma$, $p_u\oplus q_u^{-1}\oplus p_v\oplus q_v^{-1}$ and $(\gamma\cdot p_u\oplus q_u^{-1}\oplus p_v\oplus q_v^{-1})^{-1}$. Denote by $G$ the group generated by these permutations.

Firstly, lets notice that the number of orbits of $G$ acting on $\{1,\ldots,r\}$ is $|\rho(T)|$. This is because the orbits of $G$ are exactly the sets $A_B^u\cup A_B^{u^*}\cup A_B^{v}\cup A_B^{v^*}$, for each $B\in\rho(T)$. For each $i\in\{1,\ldots,k\}$, the set $A_i=A_i^u\cup A_i^{u^*}\cup A_i^{v}\cup A_i^{v^*}$ is contained in a $G$-orbit due to $\gamma$ being in $G$. For $i\neq j$, $A_i$ and $A_j$ are connected by the permutation $p_u\oplus q_u^{-1}\oplus p_v\oplus q_v^{-1}$ if and only if $i$ and $j$ belong to the same $B\in\rho(T)$. This comes from the definition of $\rho(T)$.

Secondly, as cycles in $\gamma$ represents words $(w_1,\ldots,w_k)$:
\[|\gamma|=2m_u+2m_v-\#cyc(\gamma)=2m_u+2m_v-k.\]
Thirdly, $p_u\oplus q_u^{-1}$ is a permutation on $A^u\cup A^{u^*}$, with $p_u,q_u\in Bij(A^u,A^{u^*})$. As such $p_u\oplus q_u^{-1}$ has only cycles of even length and $(p_u\oplus q_u^{-1})^2=q_u^{-1}p_u\oplus p_uq_u^{-1}$. It follows that $\#cyc(p_u\oplus q_u^{-1})=\#cyc(q_u^{-1}p_u)$. Then:
\begin{align*}
|p_u\oplus q_u^{-1}\oplus p_v\oplus q_v^{-1}|=&2m_u+2m_v-\#cyc(p_u\oplus q_u^{-1}\oplus p_v\oplus q_v^{-1})
=2m_u+2m_v-(\#cyc(q_u^{-1}p_u)+\#cyc(q_v^{-1}p_v))\\
=&m_u+m_v+|q_u^{-1}p_u|+|q_v^{-1}p_v|=m_u+m_v+L(T).
\end{align*}
Recall that $\chi(T)=V_\gamma(p,q)-(m_u+m_v)-L(T)$, where $V_\gamma(p,q)=\#cyc(\gamma\cdot p_u\oplus q_u^{-1}\oplus p_v\oplus q_v^{-1})$. So:
\[|\gamma\cdot p_u\oplus q_u^{-1}\oplus p_v\oplus q_v^{-1}|=2m_u+2m_v-V_\gamma(p,q)=m_u+m_v-\chi(T)-L(T).\]
Applying Ree's theorem, we get:
\[(2m_u+2m_v-k)+(m_u+m_v+L(T))+(m_u+m_v-\chi(T)-L(T))\geqslant 4m_u+4m_v-2|\rho(T)|.\]
After simplifications we get $\chi(T)\leqslant 2|\rho(T)|-k$ as required.
\end{proof}

\subsection{Euler estimate}

Before extending this estimation to decorated paths of albitrary length, we need the following technical result.

\begin{lemma}\label{l:technical lemma}
	Let $K\subset \s(m)$ be a subgroup, $\alpha\in K$ and $s\in\s(m)$. Let $q_0=\#Orb(K)$ and $q_s=\#Orb(\langle K,s\rangle)$. Then:
	\[2(q_0-q_s)\leqslant |s|+|\alpha s|-|\alpha|.\]
Equivalently:
	\[\#cyc(\alpha s)-|s|\leqslant \#cyc(\alpha)-2(q_0-q_s).\]
	Moreover, the diference is an even integer.
\end{lemma}
\begin{proof}
	The proof goes by induction on $|s|$. If $|s|=0$, then $s=Id$ and we have equality.
	
	Let $s=at$ with $|a|=|s|-1$ and $t$ a transposition. Let $q_a=\#Orb(\langle K,a\rangle)$. By induction hypothesis: 
	$2(q_0-q_a)\leqslant |a|+|\alpha a|-|\alpha|$. As $|at|=|a|+1$, the transposition $t$ links two cycles of $a$.
	
	Case 1: these two cycles of $a$ are in the same orbit of $\langle K,a\rangle$. Then $Orb(\langle K,s\rangle)=Orb(\langle K,a\rangle)$, and $q_s=q_a$. As $\alpha s=\alpha a\cdot t$, we have $\#cyc(\alpha s)-\#cyc(\alpha a)\in\{-1,1\}$. So $|s|+|\alpha s|-|\alpha|$ is equal to either $|a|+|\alpha a|-|\alpha|$ or to $|a|+|\alpha a|-|\alpha|+2$. 
	
	Case 2:  the two cycles of $a$ are in different orbits of $\langle K,a\rangle$. Then $\#Orb(\langle K,s\rangle)=\#Orb(\langle K,a\rangle)-1$, so $q_s=q_a-1$. As such $2(q_0-q_s)=2(q_0-q_a)+2$.
	
	Now $\alpha a\in\langle K,a\rangle$. If the transposition $t$ links two different orbits of $\langle K,a\rangle$, then it also links two cycles of $\alpha a$. As such $\#cyc(\alpha a)-\#cyc(\alpha s)=1$, so $|\alpha s|=|\alpha a|+1$. Then $|s|+|\alpha s|-|\alpha|=|a|+|\alpha a|-|\alpha|+2$ and we are done.
	
At each induction step, the difference stays the same or it increases by $2$. As such the difference is an even integer.	
\end{proof}

\begin{lemma}\label{l:from lenght 2}
For any $T\in DPh(p,q)$, $2(|\rho(T_{p,q})|-|\rho(T)|)\leqslant L(T)-L(T_{p,q})$, and the difference is an even integer.
\end{lemma}
\begin{proof}
Informally, if two sets from $\rho(T_{p,q})$ are linked in $\rho(T)$, then there exists a transposition on the path $T$, that links them and it is not contained in $p$ or $q$. So, starting from $p$, this transposition will appear at some moment on the path $T$, and dissapear until we reach $q$. As such $L(T)-L(T_{p,q})$ increases by at least $2$.
	
The formal proof goes by induction on the total number of transitions. If neither paths have an intermediate state, then  $T=T_{p,q}$ and there is nothing to prove. Otherwise, let $T_u=(p_u,\xi_1,\xi_2,\ldots,q_u)$ (interchange $u$ with $v$ if necessary). Delete $\xi_1$. If $p_u=\xi_2$, also delete $\xi_2$. Call the resulting decorated path $T_1$. It is a valid path with strictly fewer transitions and the same endpoints. If $p_u=\xi_2$, then $L(T)-L(T_1)=|p_u^{-1}\xi_1|+|\xi_1^{-1}p_u|$. In either case:
\[L(T)-L(T_1) = |p_u^{-1}\xi_1|+|\xi_1^{-1}\xi_2|-|p_u^{-1}\xi_2|.\]
By the induction hypothesis:
\[2(|\rho(T_{p,q})|-|\rho(T_1)|)\leqslant L(T_1)-L(T_{p,q}).\]
Let $K=G_{T_1}$ be the group associated to $T_1$, see Definition \ref{de:associated group}. Let $\alpha=\xi_2^{-1}p_u\oplus Id_{m_u+2m_v}$. Then $\alpha\in K$, as $\alpha=(p_u\oplus\xi_2^{-1}\oplus p_v\oplus p_v^{-1})\cdot (p_u\oplus p_u^{-1}\oplus p_v\oplus p_v^{-1})$, for example. 

Let $s=p_u^{-1}\xi_1\oplus Id_{m_u+2m_v}$. The key observation is that $\rho(\langle K,s\rangle)=\rho(T)$. We can see that $\xi_1\oplus p_u^{-1}\oplus Id_{2m_v}\in\langle K,s\rangle$, and this permutation will link whatever orbits are separated in $\rho(T_1)$, but togheher in $\rho(T)$. Also notice that $L(T)-L(T_1)=|p_u^{-1}\xi_1|+|\xi_1^{-1}\xi_2|-|p_u^{-1}\xi_2|=|s|+|\alpha s|-|\alpha|$.

 Apply Lemma \ref{l:technical lemma} to get:
\[2(|\rho(T_1)|-|\rho(T)|)\leqslant |s|+|\alpha s|-|\alpha|=L(T)-L(T_1).\]
Adding the two inequalities proves the induction step.

\end{proof}

\begin{p}\label{p:characteristic inequality}
	For any $T\in DPh$, $\chi(T)\leqslant 2|\rho(T)|-k$, and the difference is an even integer.
\end{p}
\begin{proof}
Let $T\in DPh(p,q)$. By Propostion \ref{p: chi for 2-paths}, $\chi(T_{p,q})\leqslant 2|\rho(T_{p,q})|-k$. By Lemma \ref{l:from lenght 2},  $L(T_{p,q})-L(T)\leqslant 2(|\rho(T)|-|\rho(T_{p,q})|)$. Adding the two inequalities yields the result.
\end{proof}

\section{Sharp involutions}

The goal of this section is to prove that if there exists a decorated path $T\in DPh$ such that $\rho(T)=1_S$ and $\chi(T)=2-k$, then the famility of words is \emph{planar}, meaning that the words admit conjugates that multiply to identity.

\begin{de}
A decorated path$T\in DPh$ is called \emph{sharp for $\gamma$} if $\rho(T)=1_S$ and $\chi(T)=2-k$.
\end{de}

The proof is split in two steps. Firslty, from a sharp decorated path $T_{p,q}$, we construct sharp paths of the type $T_{t,t}$. Then we show that such a path implies planarity.

\begin{lemma}
For $T_{p,q}\in DPh(p,q)$ we have:
\[V_\gamma(p,q)-L(T_{p,q})\leqslant V_\gamma(p,p)-2(|\rho(T_{p,p})|-|\rho(T_{p,q})|).\]
\end{lemma}
\begin{proof}
Apply Lemma \ref{l:technical lemma} for $K=\langle\gamma,p\oplus p^{-1}\rangle$, $\alpha=\gamma(p\oplus p^{-1})$ and $s=Id\oplus pq^{-1}$. Then $|s|=|pq^{-1}|= L(T_{p,q})$. By definition $\#cyc(\alpha)=V_\gamma(p,p)$, and $\#cyc(\alpha s)=\#cyc(\gamma(p\oplus q^{-1}))=V_\gamma(p,q)$. Also $Orb(K)=\rho(T_{p,p})$ and $Orb(K,s)=\rho(T_{p,q})$. Then:

\[V_\gamma(p,q)-L(T_{p,q})\leqslant V_\gamma(p,p)-2(|\rho(T_{p,p})|-|\rho(T_{p,q})|).\]
\end{proof}

\begin{lemma}
	For $T\in DPh(p,q)$ we have:
	\[V_\gamma(p,q)-L(T)\leqslant V_\gamma(p,p)-2(|\rho(T_{p,p})|-|\rho(T)|).\]
\end{lemma}
\begin{proof}
	By the previous Lemma $V_\gamma(p,q)-L(T_{p,q})\leqslant V_\gamma(p,p)-2(|\rho(T_{p,p})|-|\rho(T_{p,q})|)$. By Lemma \ref{l:from lenght 2}, $L(T_{p,q})-L(T)\leqslant 2(|\rho(T)|-|\rho(T_{p,q})|)$. Adding these two inequalities yields the result.
\end{proof}

\begin{te}
If $T\in DPh(p,q)$ is sharp, then $V_\gamma(p,p)-(m_u+m_v)=2|\rho(T_{p,p})|-k$.
\end{te}
\begin{proof}
Apply Propostion \ref{p:characteristic inequality} for $T_{p,p}$ to get $V_\gamma(p,p)-(m_u+m_v)\leqslant 2|\rho(T_{p,p})|-k$, as $L(T_{p,p})=0$. As $T$ is sharp, we have $V_\gamma(p,q)-(m_u+m_v)-L(T)=2-k$. Then:
\begin{align*}
2-k=&V_\gamma(p,q)-(m_u+m_v)-L(T)\leqslant V_\gamma(p,p)-(m_u+m_v)-2(|\rho(T_{p,p})|-1)\\
\leqslant&2|\rho(T_{p,p})|-k-2(|\rho(T_{p,p})|-1)=2-k.
\end{align*}
As such, we have equality in all the above inegualities, so $V_\gamma(p,p)-(m_u+m_v)= 2|\rho(T_{p,p})|-k$.
\end{proof}

Recall that for $B\subset S$, $\gamma_B$ is the product of cycles corseponding to words in $B$, $A_B^x\subset\{1,\ldots,r\}$ is the set of $x$-appearences in these words, $x\in\{u,u^*,v,v^*\}$; $A_B=\cup_xA_B^x$, so that $\gamma_B$ is a permutation on $A_B$.

Let $p=(p_u,p_v)$, where $p_u\in Bij(A_S^u,A_S^{u^*})$ and $p_v\in Bij(A_S^v,A_S^{v^*})$. Let $B\in\rho(T_{p,p})$. Then $p_u$ restricted to $A_B^u$, is a bijection from $A_B^u$ to $A_B^{u^*}$. The same way $p_v|_{A_B^v}$ is a bijection from $A_B^v$ to $A_B^{v^*}$. We denote by $p_B$ the pair $(p_u|_{A_B^u},p_v|_{A_B^v})$.

\begin{te}\label{te:sharp split}
If $T\in DPh(p,q)$ is sharp for $\gamma$, then $T_{p_B,p_B}$ is sharp for $\gamma_B$, for any $B\in\rho(T_{p,p})$.
\end{te}
\begin{proof}
Because $B\in\rho(T_{p,p})$ we have $\rho(T_{p_B,p_B})=1_B$, so $|\rho(T_{p_B,p_B})|=1$. By Proposition \ref{p:characteristic inequality}, $\chi(T_{p_B,p_B})\leqslant 2-|B|$. In order to prove that $T_{p_B,p_B}$ is sharp, we need to show that $\chi(T_{p_B,p_B})=2-|B|$. As $\sum_{B\in\rho(T_{p,p})}|B|=k$, summing up these inequalities we get:
\[\sum_{B\in\rho(T_{p,p})}\chi(T_{p_B,p_B})\leqslant 2|\rho(T_{p,p})|-k.\]

As $L(T_{p_B,p_B})=0$, we have $\chi(T_{p_B,p_B})=V_{\gamma_B}(p_B,p_B)-(|A_B^u|+|A_B^v|)$. Moreover $\sum_{B\in\rho(T_{p,p})}V_{\gamma_B}(p_B,p_B)=V_\gamma(p,p)$, $\sum_{B\in\rho(T_{p,p})}|A_B^u|=m_u$ and $\sum_{B\in\rho(T_{p,p})}|A_B^v|=m_v$. As such, by the previous Proposition:
\[\sum_{B\in\rho(T_{p,p})}\chi(T_{p_B,p_B})=V_\gamma(p,p)-(m_u+m_v)=2|\rho(T_{p,p})|-k.\]
It follows that $\chi(T_{p_B,p_B})=2-|B|$, for any $B\in\rho(T_{p,p})$.
\end{proof}

All information of a direct endpoint path, $T_{p,q}$, is stored in the permutation $p_u\oplus q_u^{-1}\oplus p_v\oplus q_v^{-1}$. If $p=q$, then this permutation is an involution.

\begin{de}
We say that $\gamma$ admits \emph{a sharp involution} $\delta$ if $\delta=p_u\oplus p_u^{-1}\oplus p_v\oplus p_v^{-1}$, and $T_{p,p}$ is sharp for $\gamma$.
\end{de}

In this section we proved that if $\gamma$ admits a sharp path $T$, then there exists a partition $\rho$ of $S$, such that  $\gamma_B$ admits a sharp involution for each $B\in\rho$. In the next section we show that a sharp involution implies planarity.

\subsection{Planar words}

\begin{de}
A family of words $w_1,\ldots,w_k$ is called \emph{planar} if they have conjugates that multiply to identity. Formally, there exists $a_1,\ldots,a_k\in\ff_2$ such that:
\[a_1w_1a_1^{-1}\cdot a_2w_2a_2^{-1}\cdot\ldots\cdot a_kw_ka_k^{-1}=e.\]
\end{de}

This is the boundary relation of a genus-zero oriented surface mapped to the bouquet of generator circles: the product of conjugates of its oriented boundary words is one. See \cite[Chapter V, Section 1]{LS77} for diagrammatic background and \cite{Cul81} for surface equations in free groups. The word-integral surface interpretation is developed in \cite{MP19}.

The associated permutation $\gamma\in\s(r)$ determines the words only cyclicly and does not remember the order of the words. However this information doesn't affect planarity.

\begin{lemma}
Let $w_1,\ldots,w_k$ be a planar family of words, and let $\vp\in\s(k)$. Then $w_{\vp(1)},\ldots,w_{\vp(k)}$ is also planar. For any $a_1,\ldots,a_k\in\ff_2$, $a_1w_1a_1^{-1},\ldots, a_kw_ka_k^{-1}$ is also planar.
\end{lemma}
\begin{proof}
We have $a_1w_1a_1^{-1}\cdot a_2w_2a_2^{-1}=(a_1w_1a_1^{-1}a_2)w_2(a_2^{-1}a_1w_1^{-1}a_1^{-1})\cdot a_1w_1a_1^{-1}$. As such, if $w_1,w_2,\ldots,w_k$ is planar, so is $w_2,w_1,\ldots,w_k$. This operation can be performed to get any consecutive transposition, and repetead to get any permutation. The second part is trivial.
\end{proof}

The last lemma shows that the following notion is well-defined.

\begin{de}
A permutation $\gamma\in\s(r)$ is called \emph{planar} if the associated family of words is planar.
\end{de}

The main result of this section is that if $\gamma$ admits a sharp involution $\delta$, then $\gamma$ is planar. The proof goes by induction on $\#cyc(\gamma)$. We record the initial step, when $\#cyc(\gamma)=1$, as a separate lemma.

\begin{lemma}
Let $\gamma$ be a one cycle permutation and	let $\delta$ be a sharp involution for $\gamma$. Then the word correspondig to $\gamma$ is equal to $e$ in $\ff_2$.
\end{lemma}
\begin{proof}
By hypothesis $\#cyc(\gamma)=1$ and $\#cyc(\delta)=r/2$. As such $\#cyc(\gamma\delta)=1+r/2$. Now $\delta$ is a product of $r/2$ disjoint transposition. Starting from the one cycle of $\gamma$, each of these transpositions changes the number of cycles by $1$ in absolute value. As such, $1+r/2$ is the maximum possible value, so any transposition of $\delta$ mush split a current cycle rather than joining two cycles.

Consequently two pairing chords of $\delta$ cannot cross. Indeed, if $(a,c)$ and $(b,d)$, two transpositions of $\delta$,  occurred in $\gamma$, in the order a, b, c, d, applying $(a,c)$ first would split the cycle with $b$ and $d$ in different pieces, and applying $(b, d)$ next would join them, contradicting maximality. Choose a pair with the smallest number of positions on its shorter interval. Any position strictly inside that interval would have to be paired inside it, producing a smaller interval. Thus the chosen pair is consecutive.

As $\delta$ links inverse letters the consecutive pair is $x,x^*$, with $x\in\{u,u^*,v,v^*\}$. The cycle $\gamma$ is $(\ldots axx^*b\ldots)$ and $\delta$ has $(xx^*)$ as a transposition. As such:
\[\gamma (xx^*)=(\ldots axb\ldots)(x^*).\]
So $\gamma\delta$ has $x^*$ as a fixed point, while $x$ is contained in some cycle of length greater than $1$. Construct $\gamma_1,\delta_1\in\s(r-2)$ by deleting the $x$ and $x^*$ occurences, so $\gamma_1=(\ldots ab\ldots)$, and $\delta_1$ has one less transposition. Then $\#cyc(\gamma_1)=1$, $\#cyc(\delta_1)=r/2-1$ and $\#cyc(\gamma_1\delta_1)=\#cyc(\gamma\delta)-1$. So $\delta_1$ is a sharp involution for $\gamma_1$ and the argument can be repeated.
\end{proof}

\begin{p}\label{p:sharp to planar}
Let $\delta$ be a sharp involution for $\gamma$. Then $\gamma$ is planar.
\end{p}
\begin{proof}
Assume $\#cyc(\gamma)=k>1$. Let $\delta$ be such that $\langle\gamma,\delta\rangle$ is transitive on $\{1,\ldots,r\}$ and:
\[\#cyc(\gamma\delta)-\#cyc(\delta)+\#cyc(\gamma)=2.\]
Let $t\in\s(r)$ be a transposition, contained in $\delta$ such that its endpoints are in different $\gamma$-cycles. 
Let $\gamma_1=\gamma t$ and $\delta_1=\delta t$. As $t$ commutes with $\delta$, we get $\gamma_1\delta_1=\gamma\delta$.

Assume $t=(a\ b)$, with $a,b\in\{1,\ldots,r\}$. Then $\gamma$ contains cycles $(A\ a)$ and $(B\ b)$, where $A$ and $B$ are some strings of poins in $\{1,\ldots,r\}$. By direct computation $\gamma_1=\gamma t$ contains the cycle $(A\ a\ B\ b)$.

Let $X=\{1,\ldots,r\}\setminus\{a,b\}$ and construct $\gamma_2,\delta_2\in\s(X)$ as follows: $\gamma_2$ has $(A\ B)$ as a cycle, and all other cycles of $\gamma_1$ are unchanged; $\delta_2$ is just $\delta_1$, with the fix poins $a,b$ removed. We show that $\delta_2$ is sharp for $\gamma_2$. 

By construction, for $x\in X$, we have $\gamma_2(x)=\gamma_1(x)$ if $\gamma_1(x)\notin\{a,b\}$; $\gamma_2(x)=\gamma_1^2(x)$ if $\gamma_1(x)\in\{a,b\}$ and $\gamma_1^2(x)\notin\{a,b\}$; and $\gamma_2(x)=\gamma_1^3(x)$ if $\gamma_1(x),\gamma_1^2(x)\in\{a,b\}$. Also, for $x\in X$, $\delta_2(x)=\delta_1(x)$.

Now, if $\gamma_1\delta_1(x)\notin\{a,b\}$, then $\gamma_2\delta_2(x)=\gamma_1\delta_1(x)$, by the construction of $\gamma_2$. Assume that $\gamma_1\delta_1(x)\in\{a,b\}$, and $(\gamma_1\delta_1)^2(x)\notin\{a,b\}$. Then $\gamma_1^2\delta_1(x)=(\gamma_1\delta_1)^2(x)$ as $a,b$ are fixed poins for $\delta_1$. So $\gamma_1^2\delta_1(x)\notin\{a,b\}$. It follows, by the construction of $\gamma_2$, that $\gamma_2\delta_2(x)=\gamma_1^2\delta_1(x)=(\gamma_1\delta_1)^2(x)$.

Same goes for the case $\gamma_1\delta_1(x),(\gamma_1\delta_1)^2(x)\in\{a,b\}$. We have $\gamma_1^2\delta_1(x)=(\gamma_1\delta_1)^2(x)$ and $\gamma_1^3\delta_1(x)=(\gamma_1\delta_1)^3(x)$. As such $\gamma_1\delta_1(x),\gamma_1^2\delta_1(x)\in\{a,b\}$ and, by the construction of $\gamma_2$, we have $\gamma_2\delta_2(x)=\gamma_1^3\delta_1(x)=(\gamma_1\delta_1)^3(x)$.

In conclusion $\gamma_2\delta_2$ is the old $\gamma_1\delta_1$ with $a$ and $b$ deleted from the cycle decomposition, i.e. the same way that $\gamma_2$ is constructed from $\gamma_1$. 

Now we can count cycles. Firstly assume that $A$ and $B$ are empty, i.e. $\gamma$ has $a$ and $b$ as fixed points. As $\delta$ has $(a\ b)$ as a cycle, because $\langle\gamma,\delta\rangle$ is transitive, it follows that $\{1,\ldots,r\}=\{a,b\}$. 
In this case $\gamma$ describes the words $x$ and $x^*$, which is a planar set of words and we are done.

If $A$ or $B$ are not trivial, then $\#cyc(\gamma_2)=\#cyc(\gamma_1)=\#cyc(\gamma)-1$. Obiously $\#cyc(\delta_2)=\#cyc(\delta)-1$. We show that $\#cyc(\gamma_2\delta_2)=\#cyc(\gamma\delta)$. The number of cycles can change only if a cycle of $\gamma\delta$ is entirely composed of $a$ and/or $b$, i.e. if $a$ or $b$ are fixed points or $(a\ b)$ is a cycle.

If $\gamma\delta(a)=a$ then $\gamma(b)=a$, which is impossible, as $a$ and $b$ are in different $\gamma$ cycles. If $(a\ b)$ is a cycle of $\gamma\delta$, then $\gamma$ has $a$ and $b$ as fixed points. The case was treated above. As such $\#cyc(\gamma_2\delta_2)=\#cyc(\gamma\delta)$, so $\#cyc(\gamma_2\delta_2)-\#cyc(\delta_2)+\#cyc(\gamma_2)=2$.

As $\langle\gamma,\delta\rangle$ is transitive, $\langle\gamma_1,\delta_1\rangle$ is transitive, as the $\gamma$-cycles that were linked by $\delta$ are now one $\gamma_1$-cycle. Then $\langle\gamma_2,\delta_2\rangle$ is transitive, except for the case when $A$ and $B$ are empty which is solved. As such $\delta_2$ is sharp for $\gamma_2$.

By induction hypothesis $\gamma_2$ is planar. Let $w_1,\ldots,w_k$ be the words associated to $\gamma$. Assume that the transposition $t$ links $w_1$ to $w_2$. Then $w_1=Ax$ and $w_2=Bx^*$, where $A,B\in\ff_2$ and $x\in\{u,u^*,v,v^*\}$. Then $\gamma_2$ coresponds to the family of words $AB,w_3,\ldots,w_k$.

As $\gamma_2$ is panar, there exists $a,a_3,\ldots,a_k\in\ff_2$ such that:
\[aABa^{-1}\cdot a_3w_3a_3^{-1}\cdot\ldots\cdot a_kw_ka_k^{-1}=e.\]
As $AB=Ax\cdot x^*(Bx^*)x$, we have:
\[aAxa^{-1}\cdot ax^*(Bx^*)xa^{-1}\cdot a_3w_3a_3^{-1}\cdot\ldots\cdot a_kw_ka_k^{-1}=e.\]
This proves that $\gamma$ is planar.
\end{proof}

\subsection{Main result}

Last sections proved that if the cummuland has order $N^{2-k}$, then the words are planar. Here we formalise this result. Recall that $\kappa_\gamma^N=\sum_{\pi\in Part(S)}\mu(\pi,1_S)I_\gamma^N(\pi)$

\begin{de}
Denote by $c_0^\gamma$ the coefficient of the term $N^{2-k}$ in $\kappa_\gamma^N$. Formally:
\[c_0^\gamma=\lim_{N\to\infty}N^{k-2}\kappa_\gamma^N.\]
\end{de}

\begin{te}\label{te:leading coefficient}
If $c_0^\gamma\neq 0$, then $\gamma$ is planar.
\end{te}
\begin{proof}
By Theorem \ref{te:formula cummuland}, $\kappa_\gamma^N=\sum_{T\in DPh,\rho(T)=1_S}(-1)^{\ve(T)}N^{\chi(T)}$. By Proposition \ref{p:characteristic inequality}, $\chi(T)\leqslant 2-k$, for any $T$ with $\rho(T)=1_S$. As such $c_0^\gamma\neq 0$ implies the existance of a sharp decorated path for $\gamma$. By Theorem \ref{te:sharp split}, the familiy of words can be partitioned into subfamilies, each admitting a sharp involution. By Propostion \ref{p:sharp to planar}, these subfamilies are planar, so $\gamma$ is planar.
\end{proof}

\begin{cor}\label{cor:obstruction}
	For a fixed family of nonempty words in $\ff_2$, for large enough $N$, the connected cumulant has a convergent expansion:
	\begin{equation}\label{eq:genus-cumulant}
		\kappa_\gamma^N=\sum_{g\ge0}c_g N^{2-k-2g}.
	\end{equation}
	If the family is not planar, then \(c_0=0\) and
	\begin{equation}\label{eq:improved}
		\kappa_\gamma^N=O(N^{-k}).
	\end{equation}
	In particular, \eqref{eq:improved} holds whenever
	\(w_1\notin\langle\!\langle{w_2,\ldots,w_k}\rangle\!\rangle\) in \(F_d\).
\end{cor}

\end{document}